\documentclass[11pt,oneside]{article}

\usepackage[utf8]{inputenc}
\usepackage[english]{babel}
\usepackage{verbatim}
\usepackage{amssymb,amsfonts,amsmath,amsthm,color,stmaryrd}
\usepackage{enumitem}
\usepackage{centernot}
\usepackage{array}

\usepackage[perpage]{footmisc}

\usepackage[scale=0.8]{geometry}

\usepackage{hyperref}

\usepackage{graphicx}

\theoremstyle{plain}
\newtheorem{theorem}{Theorem}[section]
\newtheorem{lemma}[theorem]{Lemma}
\newtheorem{proposition}[theorem]{Proposition}
\newtheorem{corollary}[theorem]{Corollary}
\newtheorem{remark}[theorem]{Remark}

\theoremstyle{definition}
\newtheorem{definition}[theorem]{Definition}

\theoremstyle{remark}

\newcounter{a}
\newenvironment{cond}{\par \refstepcounter{a}  
 {\upshape \textbf{Cond~\thea}:}}{\par}

\newcommand{\C}{\textbf{Cond}}

\newcommand{\NN}{\mathbb{N}}
\newcommand{\ZZ}{\mathbb{Z}}
\newcommand{\LL}{\mathbb{L}}

\newcommand{\ind}[1]{\mathbf{1}_{#1}}
\newcommand{\esp}[2]{E_{#1} \left[ #2 \right]}
\newcommand{\cov}[2]{\text{Cov} \left(#1,#2\right)}
\newcommand{\var}[1]{\text{Var} \left( #1 \right)}
\newcommand{\prob}[1]{P\left( #1 \right)}
\newcommand{\PCA}[1]{\mathbf{#1}} 
\newcommand{\prive}[2]{#1 \setminus #2}
\newcommand{\eqd}{\stackrel{d}=}
\newcommand{\red}[1]{{\bf #1}}
\newcommand{\prt}[1]{\alpha_{#1}}

\newcommand{\HZ}[1]{\xi_{#1}}

\newcommand{\floor}[1]{\left \lfloor #1 \right \rfloor}

\newcommand{\mes}[1]{\mathcal{M}\left(#1\right)}

\newcommand{\tr}[1]{\Phi_\PCA{#1}}

\newcommand{\BC}[2]{\text{#1}(#2)}
\newcommand{\PM}[1]{\text{PM}\left(#1\right)}
\newcommand{\bern}[1]{\mathcal{B}\left(#1\right)}

\title{\Large \bf Edge correlation function of the 8-vertex model when $\bold{a+c=b+d}$}
\author{\bf Jérôme Casse \\ Univ. Lorraine, IECL \\ UMR 7502 \\ 54500 Vandoeuvre-lès-Nancy, France}
\date{}

\begin{document}
\maketitle

\begin{abstract}
This paper is devoted to the 8-vertex model and its edge correlation function. In some particular (integrable) cases, we find a closed form of the edge correlation function and we deduce also its asymptotic. In addition, we quantify influence of boundary conditions on this function.\par
To do this, we introduce a system of particles in interaction related to the 8-vertex model. This system, studied using various tools from analytic combinatorics, random walks and conics, permits to compute the correlation function. To study the influence of boundary conditions, we involve probabilistic cellular automata of order 2.\par
\end{abstract}

\section{Introduction}
\subsection*{Vertex-models}
We start with some formal definitions. Let $K_N$ be the graph whose set of vertices is 
\begin{displaymath}
V_{N} = \{(i,j): -1 \leq i,j \leq N+1\}
\end{displaymath}
and set of edges is
\begin{align*}
E_{N} = & \{((i,j),(i,j+1)): 0 \leq i \leq N, -1 \leq j \leq N\} \\
& \cup \{((i,j),(i+1,j)): -1 \leq i \leq N, 0 \leq j \leq N\}
\end{align*}
(see Figure~\ref{fig:carreN}). The vertices
\begin{displaymath}
V^i_{N} = \{(i,j) : 0 \leq i,j \leq N\}
\end{displaymath}
are called internal vertices of $K_N$ (it is also the set of vertices adjacent to 4 edges). The edges
\begin{align*}
E^e_{N} = & \{((-1,j),(0,j)):0 \leq j \leq N\} \cup \{((N,j),(N+1,j)):0 \leq j \leq N\} \\
& \cup \{((i,-1),(i,0)):0 \leq i \leq N\} \cup \{((i,N),(i,N+1)):0 \leq i \leq N\}
\end{align*}
are called external (or boundary) edges (it is also the set of edges whose one end vertex is in $V^i_{N}$ and the other one in $\prive{V_{N}}{V^i_{N}}$).
The graph $K_N$ have $2N^2 + 2N$ edges that could be classified in two groups: the $4N$ external edges $E^e_{N}$ and the $2N^2 - 2N$ internal edges $\prive{E_{N}}{E^e_{N}}$. Each of the edges could be oriented: either ``from bottom to top'' or ``from top to bottom'' if the edge is vertical,  either ``from left to right'' or ``from right to left'' it the edge is horizontal. We call an \emph{orientation} of $K_N$, the graph $K_N$ with an orientation for every of its edges. There exists $2^{2N^2+2N}$ orientations of $K_N$ and we denote $\Omega_{N}$ the set of these orientations.\par
In the following, we call vertices of $K_N$ only its internal vertices.\par

\begin{figure}
\begin{center}
\begin{tabular}{cc}
\includegraphics{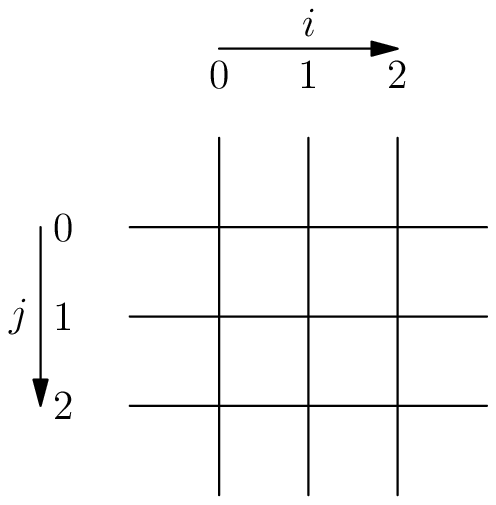} & \includegraphics{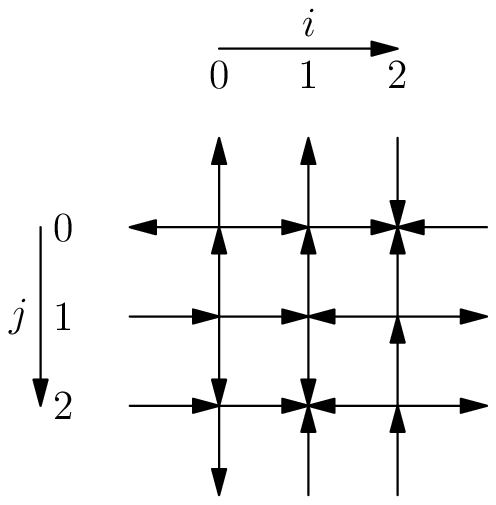} 
\end{tabular}
\end{center}
\caption{Left: $K_3$. Right: an orientation of $K_3$.}
\label{fig:carreN}
\end{figure}

In the 8-vertex model, we consider the subset $\Omega^{8}_{N} \subset \Omega_{N}$  of $K_N$'s orientations such that, around each vertex of $K_N$, there is exactly an even number (0, 2 or 4) of incoming edges. Hence, for any vertex $(i,j)$ in a $K_N$'s orientation $O \in \Omega^{8}_{N}$, the $(i,j)$'s four adjacent oriented edges are oriented like one of the eight local configurations of Figure~\ref{fig:8-vertex}.  For any $k \in \{1,\dots,8\}$, a vertex is said to be \emph{of type $k$} if its four adjacent edges are in the local configuration $k$. At each local configuration $k$, we associate a local weight $w_k$. Based on local weights $(w_k : k \in \{1,\dots,8\})$, we define a global weight (of Boltzmann type) $W$ on $\Omega^8_{N}$: let $O \in \Omega^{8}_{N}$,  the weight of the orientation $O$ is
\begin{equation}
W(O) = \prod_{k=1}^8 w_k^{n_k(O)}
\end{equation}
where $n_k(O)$ is the number of vertices of type $k$ in $O$.  From $W$ on $\Omega^{8}_{N}$, we define a probability measure $P^8_{N}$ on $\Omega^{8}_{N}$: for any $O \in \Omega^{8}_{N}$,
\begin{displaymath}
P^8_N\left(O\right) = \frac{W(O)}{Z^8_N} \text{ where } Z^8_{N} = \sum_{O' \in \Omega^8_{N}} W(O').
\end{displaymath}
The quantity $Z^8_{N}$ is called \emph{partition function} of the 8-vertex model on $K_N$. Under the probability $P^8_{N}$, the probability to get an orientation $O$ is then proportional to its weight $W(O)$.\par
When $w_7=w_8 = 0$, the 8-vertex model becomes the 6-vertex model. Orientations of the 6-vertex model with a non-zero global weight are $K_N$'s orientations whose nodes have exactly two incoming and two outgoing edges. The 6-vertex model is historically the first vertex-model introduced by Pauling in 1935~\cite{Pauling35} to study a model of ice on plane. Indeed, in the 6-vertex model, nodes represent molecules of water, and oriented edges, polarities of hydrogen bonds between these molecules. It is a model of statistical physics widely studied and we recommend~\cite[Chapter 8]{Baxter82},~\cite{Reshetikhin10},~\cite{DGHMT16} and references there in to the interesting reader.\par
The 8-vertex model is a generalization of the 6-vertex model introduced by Sutherland~\cite{Sutherland70} and Fan and Wu~\cite{FW70} in 1970. Its partition function was computed by Baxter in 1972 using Bethe's ansatz methods~\cite{Baxter72}. One important property of 8-vertex model according to 6-vertex model is that it is less dependent on boundary conditions~\cite{BKW73}. It is also related to Ising models as expressed by Baxter in~\cite[Section 10.3]{Baxter82}. Let us present some results of Baxter on 8-vertex model.

\begin{figure}
\begin{center}
\begin{tabular}{>{\centering\arraybackslash}m{.14\textwidth} | >{\centering\arraybackslash}m{.14\textwidth} >{\centering\arraybackslash}m{.14\textwidth} >{\centering\arraybackslash}m{.14\textwidth} >{\centering\arraybackslash}m{.14\textwidth} }
type & (1) & (2) & (3) & (4) \\
\hline
configuration on $K_N$ & \includegraphics{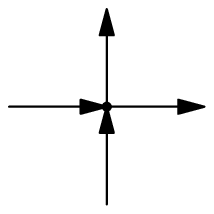} &  \includegraphics{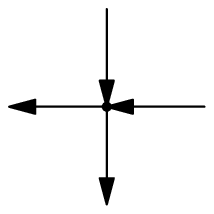} &  \includegraphics{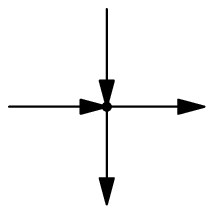} &  \includegraphics{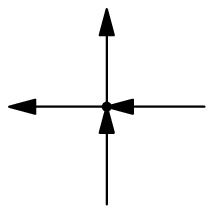} \\
\hline
configuration on $\overline{K}_N$& \includegraphics{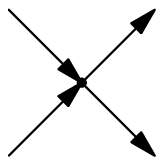} &  \includegraphics{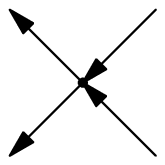} &  \includegraphics{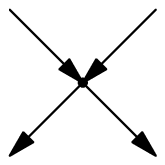} &  \includegraphics{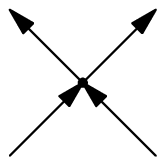} \\
\hline
\hline
type & (5) & (6) & (7) & (8) \\
\hline 
configuration on $K_N$ &  \includegraphics{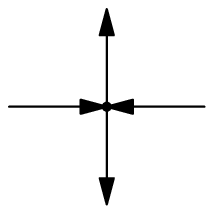} &  \includegraphics{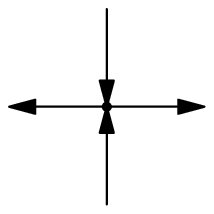} & \includegraphics{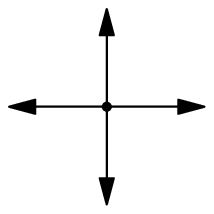} & \includegraphics{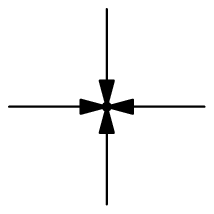} \\
\hline
configuration on $\overline{K}_N$ & \includegraphics{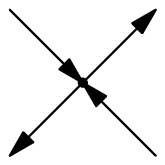} &  \includegraphics{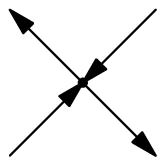} &  \includegraphics{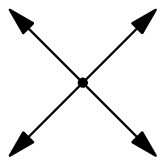} &  \includegraphics{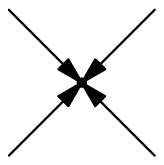}
\end{tabular}
\end{center}
\caption{The 8 available local configurations around any vertex of an orientation in $\Omega^8_N$ and their corresponding configuration (rotated by an angle $-\pi/4$) in $\overline{\Omega}^8_N$.}
\label{fig:8-vertex}
\end{figure}

Due to symmetric reasons on $K_N$, numbers of vertices of type $5$ and of type $6$ in an admissible configuration differ by less than $N$ ($|n_5-n_6| \leq N$), this permits to chose $w_5=w_6=c$ without loss of generality in asymptotic (in the asymptotic case, an interesting type $k$ of nodes is one for which we have that $n_k = \Theta \left(N^2\right)$). For similar reason, $|n_7-n_8| \leq 2N$ and, hence, we consider in the following $w_7=w_8=d$. We will also suppose that we are in a ``zero-field'' case (a classical hypothesis in a first study of a vertex-model), i.e.\ we suppose $w_1=w_2=a$ and $w_3=w_4=b$. In that case, the partition function $Z^8_{N}$ of the 8-vertex model was studied by Baxter in 1972~\cite{Baxter72} and he describes a phase transition behavior related to the value of $(a,b,c,d)$ when $N \to \infty$. He proved that the 8-vertex model has 5 different asymptotic behaviors~\cite[Section 8.10]{Baxter82}:
\begin{itemize}[topsep=0pt,itemsep=0pt,partopsep=0pt,parsep=0pt]
\item if $a > b+c+d$ (state I), then it is a ferromagnetic state, in which $N^2-o(N^2)$ vertices are either of type 1, or of type 2 a.s.\ when $N \to \infty$;
\item if $b > a+c+d$ (state II), then it is a ferromagnetic state, in which $N^2-o(N^2)$ vertices are either of type 3, or of type 4 a.s.\ when $N \to \infty$;
\item if $c > a+b+d$ (state IV), then it is an anti-ferromagnetic state, $N^2/2-o(N^2)$ vertices are of type 5 and $N^2/2-o(N^2)$ vertices are of type 6 a.s.\ when $N \to \infty$;
\item if $d > a+b+c$ (state V), then it is an anti-ferromagnetic state,  $N^2/2-o(N^2)$ vertices are of type 7 and $N^2/2-o(N^2)$ vertices are of type 8 a.s.\ when $N \to \infty$;
\item else (i.e.\ if $a,b,c,d < (a+b+c+d)/2$) (state III), then it is a disordered state, there are $\Theta(N^2)$ vertices of each type a.s. when $N \to \infty$.
\end{itemize}

Until now, we consider only the 8-vertex model on $K_N$ with free boundary conditions because external edges $E^e_N$ of $K_N$ are not constrained. In some other cases, edges of $E^e_N$ are constrained and so we do not consider all the orientations of $\Omega^8_{N}$ but a subset of them. The constraint imposed to the external edges is called a boundary condition. Some classical examples of boundary conditions are (see~\cite{BKW73}):
\begin{itemize}[topsep=0pt,itemsep=0pt,partopsep=0pt,parsep=0pt]
\item free boundary condition (FBC): no constraint on $E^e_N$;
\item periodic boundary condition (PBC): external edges of a same line or a same column are oriented in a same direction;
\item ``wall'' boundary condition (see~\cite{Z-J00}): horizontal external edges are oriented to the inside of the graph and vertical ones to the outside (see Figure~\ref{fig:wall});
\item special boundary condition (SBC)~: orientation of external edges is arbitrary imposed (```wall'' is an example).
\end{itemize} 

\begin{figure}
\begin{center}
\includegraphics{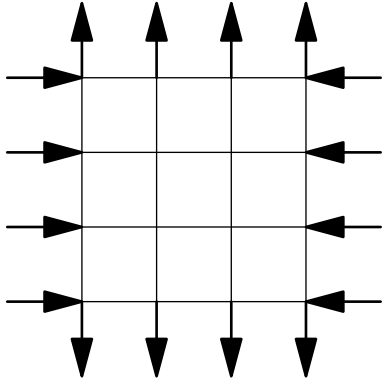}
\end{center}
\caption{$K_4$ with a ``wall'' boundary condition.}
\label{fig:wall}
\end{figure}

\medskip

In this article, main results are about 8-vertex model on $K_N$ with FBC and the following condition on local weights
\begin{equation} \label{eq:cond}
a+c=b+d.
\end{equation}
Our aim is to compute the law of orientations of two distant edges, that is the edge correlation function. This computation has been realized for 6-vertex model in some particular cases: the free fermion limit case~\cite{Sutherland68} and~the $a+c=b$ case~\cite{KDN90}. More generally, correlation functions are important subject in statistical physics, see~\cite[Chapter~10]{Sethna06} for an overview.

To obtain these results, we consider the 8-vertex model on the graph $\overline{K}_N$. Let us define it. The set of $\overline{K}_N$'s vertices is
\begin{align*}
\overline{V}_N & = \left \{(-1/2 + j,-1/2)~:~0 \leq j \leq N \right \} \\
& \qquad\bigcup\ \left \{(j,0)~:~0 \leq j \leq N-1 \right \}\\
& \qquad \bigcup_{i=0}^{N-1} \{(i/2-1/2+j,i/2+1/2)~:~0 \leq j \leq N-i \}
\end{align*}
and the set of its edges is
\begin{align*}
\overline{E}_N = & \{\ ((j-1/2,-1/2),(j,0)),((j+1/2,-1/2),(j,0)):\ 0 \leq j \leq N-1 \}
\\ & \bigcup_{i=0}^{N-1}  \{\ ((i/2-1/2+j,i/2+1/2),(i/2+j,i/2)), 
\\ & \qquad \quad ((i/2+1/2+j,i/2+1/2),(i/2+j,i/2)):\ 0 \leq j \leq N-1-i \}
\end{align*}
see Figure~\ref{fig:otriangN}). As before, we distinguish two types of vertices: internal vertices of degree 4 (called vertices in the following) and the other vertices of degree 1 or 2, and two types of edges: internal edges whose two end vertices are internal and external edges whose only one end vertex is internal. In addition, edges are numbered: the edge whose end vertices are $\{(i,t),(i',t')\}$ is indexed by $(2 \max(i,i'),2 \max(t,t'))$. And, as before, we can define 6 and 8-vertex models on $\overline{K}_N$. Definitions are the same; only local configurations change, they are rotated of an angle $-\pi/4$ (see Figure~\ref{fig:8-vertex}) from the ones in $K_N$'s case. Notations are also the same as in the $K_N$ case except that they are overlined.\par

In the following, orientations of edges are denoted by their vertical orientations: if the edge $(i,t)$ is oriented like $\nwarrow$ (if $i+t$ is even) or like $\nearrow$ (if $i+t$ is odd), then edge $(i,t)$ is said to be \emph{up-oriented} and, if the edge is oriented like $\searrow$ or $\swarrow$, it is said to be \emph{down-oriented}. This information is encoded in a state $e(i,t)$:
\begin{equation} \label{eq:defe}
e(i,t) = \begin{cases} 
1 & \text{if the edge $(i,t)$ is up-oriented}, \\ 
0 & \text{if the edge $(i,t)$ is down-oriented}.\end{cases}
\end{equation}

We introduce the following probabilistic boundary condition on the 6 and 8-vertex models on $\overline{K}_N$: orientations of edges on the top side
\begin{displaymath}
\left( e(i,0) : 0 \leq i \leq 2N-1\right)
\end{displaymath}
are distributed according to a product measure of parameter $q$, i.e.\ they are i.i.d.\ of common law the Bernouilli law of parameter $q$ (denoted $\bern{q}$): for any $i$,
\begin{displaymath}
\prob{e(i,0) = 1} = q \text{ and } \prob{e(i,0) = 0} = 1-q;
\end{displaymath}
the other external edges
\begin{align*}
\{ (i,t) : t-i=1,\ 1 \leq t \leq N \} \cup \{(i,t) : t+i = 2N,\ 1 \leq t \leq N \} 
\end{align*} 
are free oriented. We call this boundary condition half product measure of parameter $q$ (denoted $\BC{HPMBC}{q}$ in the following). Formally, the 8-vertex model on $\overline{K}_N$ with $\BC{HPMBC}{q}$ defines a probability measure $\overline{P}^{8,q}_{N}$ on $\overline{\Omega}^8_{N}$ in the following way: for any $O \in \overline{\Omega}^8_{N}$,
\begin{equation} \label{eq:HPMBC}
\overline{P}^{8,q}_{N}(O) = \sum_{B \in \{0,1\}^{2N}} \prod_{i=0}^{2N-1} q^{e(i,0)} (1-q)^{1-e(i,0)} \frac{a^{n_1+n_2} b^{n_3+n_4} c^{n_5+n_6} d^{n_7+n_8}}{\overline{Z}^{8,B}_{N}} \ind{(e(i,0): 0 \leq i \leq 2N-1)=B}
\end{equation}
where $\overline{Z}^{8,B}_{N}$ is the partition function of the 8-vertex model on $\overline{K}_N$ with SBC $B$ on edges $((i,0): 0 \leq i \leq 2N)$ and FBC on other edges of $\overline{E}^e_N$. Denoting the subset of $\overline{\Omega}^8_{N}$ such that edges $((i,0): 0 \leq i \leq 2N)$ are oriented as $B$ by $\overline{\Omega}^{8,B}_{N}$,
\begin{equation}
\overline{Z}^{8,B}_{N} = \sum_{O' \in \overline{\Omega}^{8,B}_{N}} a^{n_1(O')+n_2(O')} b^{n_3(O')+n_4(O')} c^{n_5(O')+n_6(O')} d^{n_7(O')+n_8(O')}.
\end{equation}
\par

\begin{figure}
\begin{center}
\includegraphics{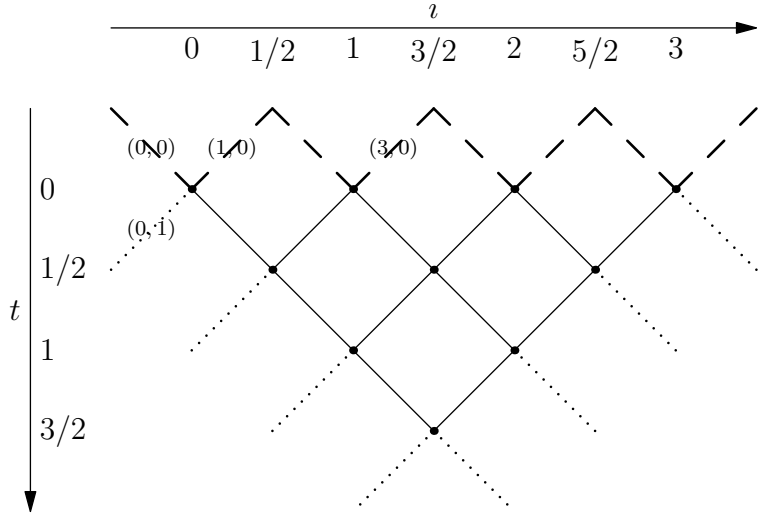}
\end{center}
\caption{The graph $\overline{K}_4$. On axis: vertices numbering. Boundary conditions: dashed edges are i.i.d.\ (for the orientation up/bottom) and dotted edges are free. On edges: edges' numbering.}
\label{fig:otriangN}
\end{figure}

Our aim is to describe properties of 6 and 8-vertex models on $\overline{K}_N$ in the thermodynamic limit (when $N \to \infty$) under the constraint $a+c=b+d$. A first remarkable property is that, when $a+c=b+d$, 8-vertex models on $\overline{K}_N$ with $\BC{HPMBC}{1/2}$ or with $\text{FBC}$ are the same. 
\begin{proposition} \label{prop:8bc}
For any $N$, if $a+c=b+d$, then $\overline{P}^{8,1/2}_{N} = \overline{P}^{8}_{N}$.
\end{proposition}
\medskip
We define now the graph $\overline{K}_\infty$ on the half-plan $\ZZ \times \NN$ and a probabilistic boundary condition on this graph. Later (in Proposition~\ref{prop:limitlaw}) this graph and its boundary condition will appear as the limit of the sequence of graphs $(\overline{K}_N: N \geq 1)$ with $\BC{HPMBC}{q}$ when $N \to \infty$ (in a sense that we will precise). The set of $\overline{K}_\infty$'s vertices is
\begin{displaymath}
\overline{V}_\infty = \{(i-1/2,t-1/2),\ (i,t): i \in \ZZ, t \in \NN\}
\end{displaymath}
and its set of edges is
\begin{align*}
\overline{E}_\infty = &\ \{((i-1/2,t-1/2),(i,t)),\ ((i-1/2,t+1/2),(i,t)), \\
& \quad ((i+1/2,t-1/2),(i,t)),\ ((i+1/2,t+1/2),(i,t)): i \in \ZZ, t \in \NN\}
\end{align*}
(see Figure~\ref{fig:ocarreI}). As before, there are two types of vertices: internal vertices (called vertices in the following) of degree 4 and the other vertices $\{(i-1/2,-1/2): i \in \ZZ\}$ of degree 2. As for $\overline{K}_N$, edges are numbered: edge whose end vertices is $\{(i,t),(i',t')\}$ is indexed by $(2 \max(i,i'),2 \max(t,t'))$. Edges $\{(i,0) : i \in \ZZ\}$ are external and others are internal. On this graph, we call product measure boundary condition of parameter $q$ ($\BC{PMBC}{q}$) the probabilistic boundary condition such that $(e(i,0): i \in \ZZ)$ are i.i.d.\ of common law $\bern{q}$.

\begin{figure}
\begin{center}
\includegraphics{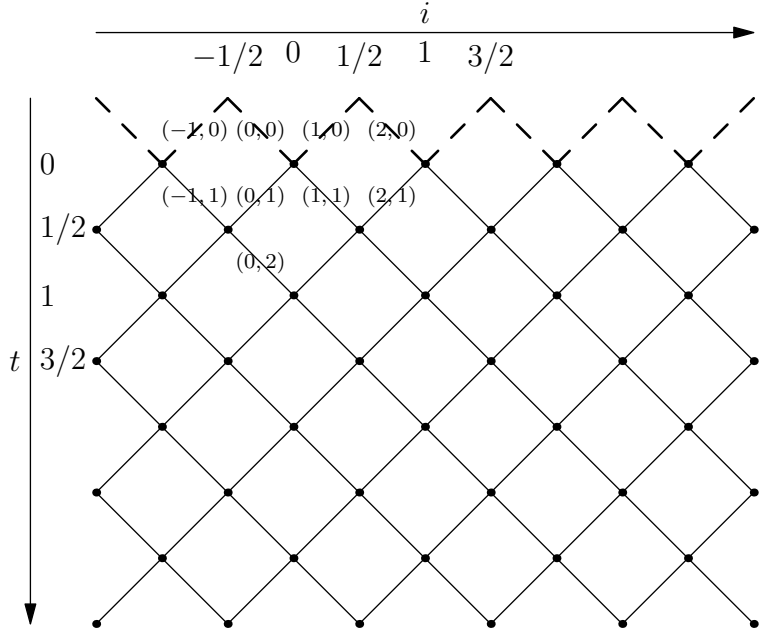}
\end{center}
\caption{The graph $\overline{K}_\infty$. On the axis: vertices numbering. On edges: edges numbering. Boundary condition $\BC{PMBC}{q}$: vertical orientations of dashed edges $(e(i,0) : i \in \NN)$ are i.i.d.}
\label{fig:ocarreI}
\end{figure}

Let $\overline{\Omega}^8_\infty$ be the set of $\overline{K}_\infty$'s orientations such that any vertex of an orientation $O \in \overline{\Omega}^8_\infty$ has $0$, $2$ or $4$ incoming edges. We can define a probability measure $\overline{P}^8_{\infty}$ on $\overline{\Omega}^8_\infty$ associated to the 8-vertex model on $\overline{K}_\infty$ with $\BC{PMBC}{1/2}$. In general case (for any $a$, $b$, $c$ and $d$), this measure must be seen as limit law of probability measures $\overline{P}^{8,1/2}_N$ when $N \to \infty$ in a certain sense. In the case $a+c=b+d$, there is a simpler way to prove its existence by considering the law $\displaystyle \mathcal{L}\left(\PM{\frac{1}{2}};\frac{a}{a+c},\frac{b}{b+d}\right)$, defined below.

\subsection*{The law $\mathcal{L}(\mu;p,r)$ on $\overline{\Omega}_{\infty}$}
In~\cite{KDN90}, Kandel, Domany and Nienhuis defined a law on $\overline{\Omega}^6_\infty$ (the set of orientations of $\overline{K}_\infty$ with exactly 2 incoming edges around each vertex) as the law of a Markov chain whose state space is $\{0,1\}^\ZZ$. This law is in fact the limit law of the 6-vertex model on $\overline{K}_N$ with $\BC{HPMBC}{q}$ when $N \to \infty$. Here, we generalize their idea to define laws $\mathcal{L}(\mu;p,r)$ on $\overline{\Omega}^8_{\infty}$, whose one specification is $\overline{P}^8_\infty$.
\medskip
\begin{definition}[Law $\mathcal{L}(\mu;p,r)$] \label{def:loiqpr}
Let $\mu$ be any probability measure on $\{0,1\}^\ZZ$. Let $p,r \in [0,1]$, we define the law $\mathcal{L}(\mu;p,r)$ on $\overline{\Omega}^8_{\infty}$ by:
\begin{itemize}[topsep=0pt,itemsep=0pt,partopsep=0pt,parsep=0pt]
\item Law of orientations $(e(i,0): i \in \ZZ)$ of edges on the first line is $\mu$.

\item For any even $t$, starting with orientations $(e(i,t): i \in \ZZ)$ on line $t$, we obtain orientations $(e(i,t+1): i \in \ZZ)$ on line $t+1$ by the following way: for any $i \in \ZZ$, orientations of pair $(e(2i,t+1),e(2i+1,t+1))$ depends only on pair $(e(2i,t),e(2i+1,t))$ and local transition probabilities are, for any $i\in \ZZ$, any $k \in \{0,1\}$,
\begin{align*}
& \prob{\ (e(2i,t+1),e(2i+1,t+1)) = (k,k)\ |\ (e(2i,t),e(2i+1,t)) = (k,k)\ } = r,\\
& \prob{\ (e(2i,t+1),e(2i+1,t+1)) = (1-k,1-k)\ |\ (e(2i,t),e(2i+1,t)) = (k,k)\ } = 1-r,\\
& \prob{\ (e(2i,t+1),e(2i+1,t+1)) = (1-k,k)\ |\ (e(2i,t),e(2i+1,t)) = (k,1-k)\ } = p,\\
& \prob{\ (e(2i,t+1),e(2i+1,t+1)) = (k,1-k)\ |\ (e(2i,t),e(2i+1,t)) = (k,1-k)\ } = 1-p
\end{align*}
and local transitions from pair $(e(2i,t),e(2i+1,t))$ to pair $(e(2i,t+1),e(2i+1,t+1))$ are independent of one another, i.e.\ for any $i_1,i_2 \in \ZZ$ such that $i_1 < i_2$, for any $(k_{2 i_1}, k_{2 i_1+1}, \dots , k_{2 i_2}  , k_{2 i_2+1}) \in \{0,1\}^{2(i_2-i_1+1)}$,
\begin{align*}
& \prob{\ (e(j,t+1) = k_j : 2 i_1 \leq j \leq 2 i_2 +1) \ |\ (e(i,t) : i \in \ZZ) \ } \\
& = \prod_{i=i_1}^{i_2} \prob{\ (e(2i,t+1),e(2i+1,t+1)) = (k_{2i},k_{2i+1})\ |\ (e(2i,t),e(2i+1,t))\ }.
\end{align*}

We denote by $T_0$ this operator on $\mes{\{0,1\}^\ZZ}$, the set of $\{0,1\}^\ZZ$'s probability measures:
\begin{equation}
T_0((e(i,t): i \in \ZZ)) = (e(i,t+1): i \in \ZZ).
\end{equation}

\item For any odd $t$, transition is the same as in case even $t$ with the difference that we consider pairs of edges of abscissas $(2i-1,2i)$ instead of pairs of edges of abscissas $(2i,2i+1)$. We denote by $T_1$ this operator.
\end{itemize}

Local transitions of these two operators are illustrated on Figure~\ref{fig:op}.

Finally, the law $\mathcal{L}(\mu;p,r)$ on $\overline{\Omega}^8_\infty$ is the law of $(e(i,t):i \in \ZZ, t \in \NN)$.
\end{definition}
\medskip

\begin{figure}
\begin{center}
\begin{tabular}{>{\centering\arraybackslash}m{.15\textwidth} |ccc|ccc}
Initial state at time $t$ & \multicolumn{3}{c|}{\raisebox{-5mm}{\includegraphics{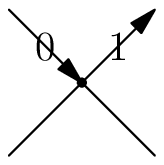}}} & \multicolumn{3}{c}{\raisebox{-5mm}{\includegraphics{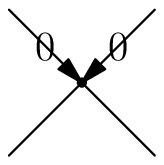}}} \\
\hline
Final state at time $t+1$ & \raisebox{-5mm}{\includegraphics{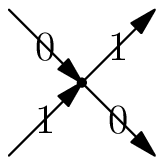}} & or & \raisebox{-5mm}{\includegraphics{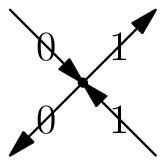}} &   \raisebox{-5mm}{\includegraphics{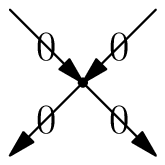}} & or & \raisebox{-5mm}{\includegraphics{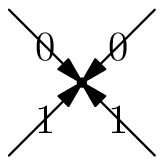}} \\
& w.p. $p$ & & w.p. $1-p$ & w.p. $r$ & & w.p. $1-r$
\end{tabular}
\end{center}
\caption{Operators $T_0$ and $T_1$. In all case, the left-up arrow is indexed by $(i,t)$ with $i+t$ even. Orientation on edge correspond to a choice of $k=0$.}
\label{fig:op}
\end{figure}

In the following, the considered measure $\mu$ will be the product measure of parameter $1/2$ (denoted $\PM{1/2}$) and, sometimes, of parameter $q \in [0,1]$ (denoted $\PM{q}$). To fix notations, $(e_i: i \in \ZZ)$ is distributed according to $\PM{q}$, if $(e_i:i \in \ZZ)$ are i.i.d.\ and $e_0 \sim \bern{q}$.\par

If $r=1$ and $\mu = \PM{q}$, we recover the result of~\cite{KDN90}. For any value of $r$, we get the following generalization: 
\begin{proposition} \label{prop:limitlaw}
For any $N$, for any subset $\overline{S}_N = (\overline{V}_N,\overline{E}_N)$ of $\overline{K}_\infty$ isomorphic to $\overline{K}_N$. Let $O \sim \mathcal{L} \left( \PM{\frac{1}{2}};\frac{a}{a+c},\frac{b}{b+d} \right)$ with $a+c=b+d$. Then, the law of oriented edges of $O$ in the subset $\overline{S}_N$ ($O|_{\overline{S}_N}$) is distributed as $\overline{P}^8_N$.
\end{proposition}\par
For this reason, in the following, $\displaystyle \mathcal{L} \left( \PM{\frac{1}{2}};\frac{a}{a+c},\frac{b}{b+d} \right)$ is denoted by $\overline{P}^8_\infty$. Moreover, we obtain that
\begin{proposition} \label{prop:SKN}
Let $S_N = (V_N,E_N)$ be any subset of $\overline{K}_\infty$ isomorphic to $K_N$ rotated by an angle $-\pi/4$ (see Figure~\ref{fig:SKN}). Let $O \sim \overline{P}^8_\infty$ with $a+c=b+d$, then the law of $O_{|S_N}$ ($O$ restricted to edges in $S_N$ and rotated by an angle $\pi/4$) is $P^8_{N}$.
\end{proposition}\par

\begin{figure}
\begin{center}
\includegraphics{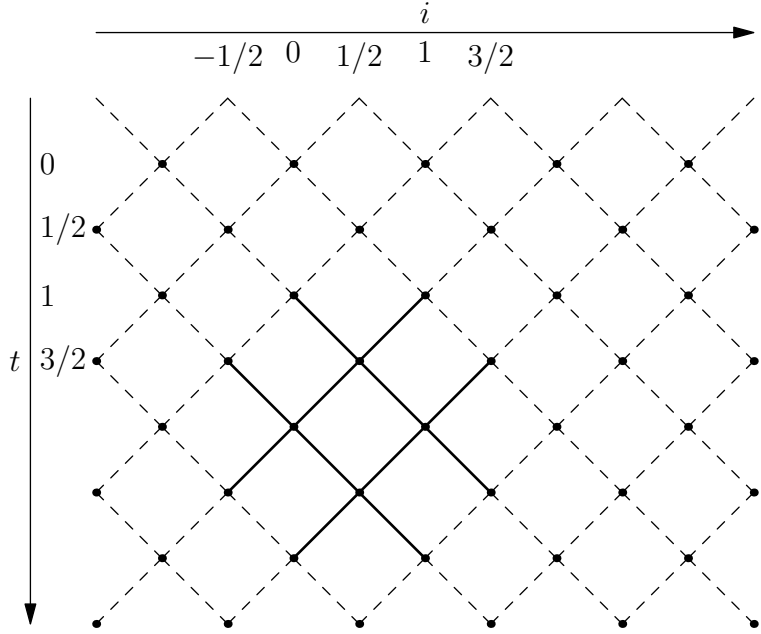}
\end{center}
\caption{In full line, a subset $\text{S}_2$ of $\overline{K}_\infty$ isomorphic to $K_2$.}
\label{fig:SKN}
\end{figure}

The reason why we consider only $\mu=\PM{1/2}$ is due to the fact that $\PM{1/2}$ has remarkable properties according to the Markov chain on $E^{\ZZ}$ with operators $T_0$ and $T_1$ used in Definition~\ref{def:loiqpr}. Indeed, it is an invariant law of this Markov chain.
\begin{proposition} \label{cor:inv2}
Let $O \sim \mathcal{L}(\PM{1/2};p,r)$. Then, for any $t$, $(e(i,t): i\in \ZZ) \sim \PM{1/2}$.
\end{proposition} 

And, moreover, if $(p,r) \in (0,1)^2$ and $p+r \neq 1$, it is the unique invariant law and the Markov chain is ergodic.
\begin{proposition} \label{prop:ergo}
Let $\mu$ be any measure on $\{0,1\}^\ZZ$ and let $(p,r) \in (0,1)^2$. Let $O \sim \mathcal{L}(\mu;p,r)$. Then, by denoting $\mu_t$ the law of $(e(i,t): i \in \ZZ)$,
\begin{displaymath}
\mu_t \stackrel{(l)}{\to} \PM{1/2} \text{ as } t \to \infty.
\end{displaymath}
\end{proposition}

Proposition~\ref{cor:inv2} generalizes the following one of~\cite{KDN90} about $6$-vertex model.
\begin{proposition}[\cite{KDN90}] \label{cor:inv3}
Suppose that $a+c=b$. Let $O \sim \overline{P}^{6,q}_\infty = \mathcal{L}(\PM{q};\frac{a}{a+c},1)$ (limit law of $\overline{P}^{6,q}_N$ when $N \to \infty$). Then, for any $t$, $(e(i,t): i\in \ZZ) \sim \PM{q}$.
\end{proposition}

Propositions~\ref{cor:inv2} and~\ref{cor:inv3} are easy to prove by coming back to Definition~\ref{def:loiqpr}. Proposition~\ref{prop:ergo} is more complicated and is proved in Section~\ref{sec:PCA} using new results about probabilistic cellular automata.\par

\subsection*{Edge correlation function of the 8-vertex model when $a+c=b+d$ and its special case the 6-vertex model when $a+c=b$.}
Let $\mu$ be any probability distribution on $\{0,1\}^\ZZ$. Let $\displaystyle O \sim \mathcal{L} \left(\mu;\frac{a}{a+c},\frac{b}{b+d}\right)$. The edge correlation function is the function $C((i,t);(i',t'))$ defined by, for any $(i,t), (i',t') \in \overline{E}_\infty$,
\begin{align}
C((i,t);(i',t')) & = \frac{\cov{e(i,t)}{e(i',t')}}{\sqrt{\var{e(i,t)}}  \sqrt{\var{e(i',t')}}}\\
& = \frac{\esp{}{e(i,t) e(i',t')} - \esp{}{e(i,t)} \esp{}{e(i',t')}}{\sqrt{\var{e(i,t)}}  \sqrt{\var{e(i',t')}}}. \nonumber
\end{align} 
We can remark that, as $e(i,t) \in \{0,1\}$ for any $(i,t)$, knowing $C((i,t),(i',t'))$ is equivalent to knowing the joint law of $e(i,t)$ and $e(i',t').$

Our main objective is to determine $C_8$ that is the edge correlation function $C$ of the 8-vertex model when $a+c=b+d$ and with FBC, i.e. when $O \sim \overline{P}^8_\infty$.

First of all, $C_8$ has some invariance properties:
\begin{proposition}  \label{prop:CInvT}
For any $(i,t), (i',t') \in \overline{E}_\infty$,
\begin{align} 
C_8((i,t);(i',t')) & = C_8((i',t');(i,t)) \label{eq:CInvT0}, \\
C_8((0,t);(i',t')) & = C_8((1,t);(1-i',t'))  \label{eq:CInvT1}, \\
C_8((i,t);(i',t')) & = C_8((i+2,t);(i'+2,t')) \label{eq:CInvT2}, \\
C_8((i,t);(i',t')) & = C_8((i+1,t+1);(i'+1,t'+1)) \label{eq:CInvT3}.
\end{align}
\end{proposition}
This is a consequence of Proposition~\ref{cor:inv2}. Proposition~\ref{prop:CInvT} permits to determine $C_8$ for any values $(i,t)$ and $(i',t')$ if the set $\{C_8((0,0);(i,t)): (i,t) \in \overline{E}_\infty)\}$ is known. Hence, in the following, we determine and denote $C_8(i,t) = C_8((0,0);(i,t))$.\par

For the 6-vertex model under similar conditions, these invariant properties was already proved in~\cite{KDN90}. And, moreover, the edge correlation function of the 6-vertex model under some conditions was evaluated in the same article:
\begin{theorem}[\cite{KDN90}] \label{thm:KDN}
For any $q \in (0,1)$, let $O \in \overline{\Omega}^6_\infty$ distributed according to $\overline{P}^{6,q}_\infty$ (the limit law of $\overline{P}^{6,q}_N$, laws of 6-vertex model on $\overline{K}_N$ with $\BC{PMBC}{q}$). If $a+c=b$, the edge correlation function $C(i,t)$ is 
\begin{equation}
C(i,2t) = \begin{cases} \displaystyle \frac{1}{2^{2t}} \binom{2t-1}{(2t-i-\Delta(i))/2} & \text{if } 2t \geq i + \Delta(i),\\
0 & \text{else},
\end{cases}
\end{equation}
with $\Delta(i) = \begin{cases} 1 & \text{if $i$ is odd,} \\ 2 & \text{if $i$ is even.} \end{cases}$

Moreover, for any $i$, when $t \to \infty$,
\begin{equation} \label{eq:asyKDN}
C(i,2t) = \Theta \left( t^{-1/2} \right).
\end{equation}
\end{theorem}\par
In their article, they consider two lines by two lines, that's why the edge correlation function is the one of $C(i,2t)$ instead of $C(i,t)$.\par

In our paper, the main objective is to give the edge correlation function $C_8(i,t)$ of the 8-vertex model on $\overline{K}_\infty$ with FBC and $a+c=b+d$. Just before to state the main theorem, we introduce some notations that are used all along the paper:
\begin{equation}
p = \frac{a}{a+c},\ r = \frac{b}{b+d} \text{ and}
\end{equation}
\begin{equation} \label{eq:DDP}
\Delta = 1 - (p+r),\ D = r-p ,\ P = (2p-1)(2r-1).
\end{equation}

\begin{theorem} \label{thm:Cit}
The edge correlation function $C_8$ of the 8-vertex model on $\overline{K}_\infty$ with FBC and $a+c=b+d$ (i.e. when $O \sim \overline{P}^8_\infty$) is:
\begin{itemize}[topsep=0pt,itemsep=0pt,partopsep=0pt,parsep=0pt]
\item if $i+t$ is odd,
\begin{equation}
C_8(i,t) = (-1)^{t+1} D \sum_{k=0}^{\frac{t-1-|i|}{2}} (-1)^k \binom{t-1-k}{k,\frac{t-1+i}{2}-k,\frac{t-1-i}{2}-k} \Delta^{t-1-2k} P^k; \label{eq:Citi}
\end{equation}

\item if $t$ is even and $i=0$,
\begin{equation}
C_8(0,t) = \sum_{k=0}^{t/2} (-1)^{k} \binom{t-1-k}{\frac{t}{2}-k} \binom{t/2}{k} \Delta^{t-2k} P^k; \label{eq:Ci0tp}
\end{equation}

\item if $i+t$ is even and $i<0$,
\begin{equation}
C_8(i,t) = \left( \sum_{k=0}^{\frac{t+i}{2}-1} (-1)^{t+k} \binom{t-1-k}{\frac{t-i}{2}-k} \binom{\frac{t+i}{2}}{k} \Delta^{t-2k} P^k \right)+ (-1)^{\frac{t-i}{2}} \binom{\frac{t-i}{2}-1}{\frac{t+i}{2}-1} \Delta^{-i} P^{\frac{t+i}{2}};
\end{equation}

\item if $i+t$ is even and $i>0$,
\begin{equation}
C_8(i,t) = \left(\sum_{k=0}^{\frac{t-i}{2}-1} (-1)^{t+k} \binom{t-1-k}{\frac{t-i}{2}-k} \binom{\frac{t+i}{2}}{k} \Delta^{t-2k} P^k\right) + (-1)^{\frac{t+i}{2}} \binom{\frac{t+i}{2}}{\frac{t-i}{2}} \Delta^{i} P^{\frac{t-i}{2}}.
\end{equation}
\end{itemize}
\end{theorem}
\par
\begin{remark} \label{rem:EASY}
There are some particular cases for which $C_8$ is simpler:
\begin{itemize}[topsep=0pt,itemsep=0pt,partopsep=0pt,parsep=0pt]
\item if ($i \geq 0$ and $t \leq i-1$) or ($i \leq -1$ and $t \leq -i$) (i.e. the edge $(i,t)$ is not in a kind of cone starting from $(0,0)$, see Figure~\ref{fig:cone}), then $C_8(i,t)=0$.
\item if $a=d$ and $b=c$ (i.e. $p = 1- r$, and so $\Delta = 0$), then 
\begin{displaymath}
C_8(i,t) = \begin{cases}
0 & \text{if } i \neq 0, \\
(1-2p)^t & \text{if } i=0.
\end{cases}
\end{displaymath}
\item if $a=b$ and $c=d$ (i.e. $p=r$, and so $D=0$), then
\begin{displaymath}
C_8(i,t) = \begin{cases}
0 & \text{if } i \neq t, \\
(2p-1)^t & \text{if } i=t.
\end{cases}
\end{displaymath}
\item if $a=c$ (i.e. $p=1/2$, and so $P=0$), then
\begin{displaymath}
C_8(i,t) = \begin{cases}
\left(-1/2\right)^t\ \binom{t-1}{(t-i)/2}\ (1-2r)^t& \text{if $i$ is even},\\
(-1/2)^{t}\ \binom{t-1}{(t-1-i)/2}\ (1-2r)^{t} & \text{if $i$ is odd}.
\end{cases}
\end{displaymath}
\item if $b=d$ (i.e. $r=1/2$, and so $P=0$), then
\begin{displaymath}
C_8(i,t) = \begin{cases}
(-1/2)^t\ \binom{t-1}{(t-i)/2}\ (1-2p)^t& \text{if $i$ is even},\\
-(-1/2)^{t}\ \binom{t-1}{(t-1-i)/2}\ (1-2p)^{t} & \text{if $i$ is odd}.
\end{cases}
\end{displaymath}

\end{itemize}
\end{remark}

\begin{figure}
\begin{center}
\includegraphics{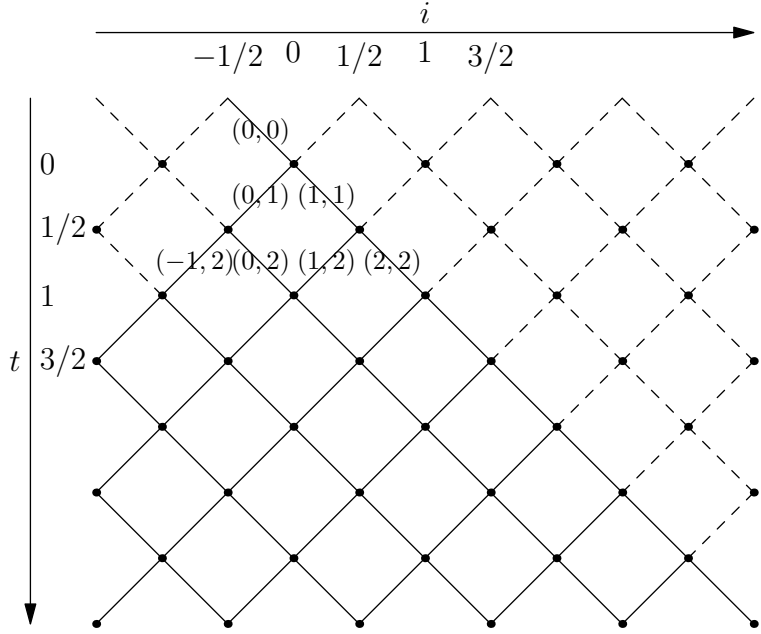}
\end{center}
\caption{The ``influence cone'' of edge $(0,0)$. In full line: the set of edges for which $C_8(i,t) \neq 0$ in general. In dashed line: the set of edges for which $C_8(i,t) = 0$ when $a+c=b+d$.}
\label{fig:cone}
\end{figure}

Moreover, we obtain the asymptotic of $C_8(i,t)$ for any $i$ when $t \to \infty$.
\begin{theorem} \label{thm:asympt}
For any $i$,
\begin{itemize}[topsep=0pt,itemsep=0pt,partopsep=0pt,parsep=0pt]
\item if $a=d$ and $b=c$ (i.e.\ $p+r=1$), for any $t$,
\begin{displaymath}
C_8(i,t) = \begin{cases}
0 & \text{if } i \neq 0, \\
(1-2p)^t & \text{if } i=0.
\end{cases}
\end{displaymath}

\item else, when $t \to \infty$,
\begin{equation}
C_8(i,t) = O \left(\frac{\lambda(p,r)^t}{\sqrt{t}} \right)
\end{equation}
with 
\begin{equation} \label{eq:lambda}
\lambda(p,r) = \max \left( |1-2p|,|1-2r|\right).
\end{equation}
\end{itemize}
\end{theorem}

For $d=0$, we find~\eqref{eq:asyKDN}, the asymptotic result of~\cite{KDN90}, that is a square-root decreasing of the edge correlation function in the 6-vertex model. In generic case of 8-vertex model when $a,b,c,d$ are all different of $0$, decreasing becomes exponential. 

In Theorem~\ref{thm:Cit} and~\ref{thm:asympt}, the regime is supposed to be stationary. When the regime is not stationary, i.e. when we start with any initial law $\mu$ at time $0$, we can obtain bounds on $C((0,0);(i,t))$.
\begin{proposition} \label{prop:vitesse}
Let $\mu$ be any probability measure on $\{0,1\}^\ZZ$ and let $O \sim \mathcal{L}(\mu;p,r)$. If $a+c=b+d$, then the edge correlation function $C$ satisfies, for any $i$, any $t$,
\begin{equation}
\left| \sqrt{\frac{\var{e(i,t)}}{\var{e(0,0)}}}\, C((0,0);(i,t)) - C_8(i,t) \right| \leq 2 \left(\lambda(p,r)^{t-1-\lfloor t/2 \rfloor} + \lambda(p,r)^{\lfloor t/2 \rfloor} - \lambda(p,r)^{t-1}\right)
\end{equation}
with $\lambda(p,r) = \max \left( |1-2p|,|1-2r|\right)$.
\end{proposition}

\subsection*{Content}
Section~\ref{sec:lqpr} is about links between 6 and 8-vertex models when $a+c=b+d$ and laws $\mathcal{L}(\mu;p,r)$. In particular, we express, in those cases, partition functions to prove Propositions~\ref{prop:8bc},~\ref{prop:limitlaw} and \ref{prop:SKN}.\par

In Section~\ref{sec:Cit}, Theorem~\ref{thm:Cit}, the main theorem of this paper, is proved. First, in Section~\ref{sec:P8V}, we establish and prove Proposition~\ref{prop:C8V} that relate $C_8(i,t)$ to a rational fraction. Proof of this proposition is based on a fundamental lemma (Lemma~\ref{lem:couplage}) that permits to understand precisely behaviors of correlations in 8-vertex model when $a+c=b+d$. Then, in Section~\ref{sec:extract}, we extract coefficients of this rational fraction to end the proof of Theorem~\ref{thm:Cit}. And, in Section~\ref{sec:vitesse}, we discuss about the influence of the boundary conditions and prove Proposition~\ref{prop:vitesse}.\par

Section~\ref{sec:asympt} concerns asymptotic of $C_8(i,t)$ and Theorem~\ref{thm:asympt} is proved. In Section~\ref{sec:asympt0}, case $r=0$ is done using properties on random walks. Then, in Section~\ref{sec:asymptg}, general case is proved using both results of Section~\ref{sec:asympt0} and Theorem~\ref{thm:Cit}.

Section~\ref{sec:PCA} is dedicated to some discussions about links between vertex models, colorings of plan and probabilistic cellular automata. In Section~\ref{sec:PCADT}, we introduce and define triangular probabilistic cellular automata (TPCA, a new type of PCA at the best knowledge of the author) and we establish theorems about their invariant probability distributions. In Section~\ref{sec:PCA8V}, we present a family of TPCA that permits to obtain laws $\mathcal{L}(\mu;p,r)$. Due to that, we obtain an alternative proof of Proposition~\ref{cor:inv2} and, most important, we get a proof of Proposition~\ref{prop:ergo}. In Section~\ref{sec:PCA6V}, we present a family of TPCA to obtain laws $\mathcal{L}(\mu;p,1)$ and we obtain an alternative proof of Proposition~\ref{cor:inv3}. In Section~\ref{sec:PACP}, theorems and properties stated in Section~\ref{sec:PCADT},~\ref{sec:PCA8V} and~\ref{sec:PCA6V} are proved.

Finally, in Section~\ref{sec:concl}, we conclude this article giving some additional commentaries on vertex models.

\section{From $\overline{K}_N$ to $\overline{K}_\infty$: laws $\mathcal{L}(\mu;p,r)$} \label{sec:lqpr}
The major aim of this section is to prove Propositions~\ref{prop:8bc}, ~\ref{prop:limitlaw} and~\ref{prop:SKN}. These propositions are in fact consequences of properties of partition functions of the 8-vertex model with $a+c=b+d$ and FBC on $\overline{K}_N$ and on $K_N$. 

\subsection{Partition function on $\overline{K}_N$}
\begin{lemma} \label{lem:pf8}
If $a+c=b+d$, then for any $N$, the partition function $\overline{Z}^{8}_{N}$ of the 8-vertex model on $\overline{K}_N$ with $\text{FBC}$ is
\begin{equation}
\overline{Z}^{8}_{N} = 2^{2N} (a+c)^{N(N+1)/2}
\end{equation}
\end{lemma}

\begin{proof}
First of all, for $N=1$, we have that
\begin{equation}
\overline{Z}^{8}_{1} = a + a + b + b + c + c + d + d = 2^2(a+c).
\end{equation}

Now, let define $L_N$ the subgraph of $\overline{K}_N$ that contains only the internal nodes $\{(i,0) : 0 \leq i \leq N-1\} \subset \overline{V}_N$ and their adjacent edges, numbered $\{(i,t): 0 \leq i \leq 2N-1, t \in \{0,1\}\}$. Then, the partition function of the 8-vertex model on $L_N$ with $\text{FBC}$ is
\begin{equation}
Z^{8}_{L_N} = 2^{2N}(a+c)^N.
\end{equation}
That can be proved by recurrence on $N$, indeed
\begin{align*}
Z^{8}_{L_{N+1}} & = \sum_{O \in \Omega^8_{L_{N+1}}} \prod_{i=0}^{N+1} w_{\text{type}_{O}(i,0)} \\
& = \left( \sum_{O \in \Omega^8_{L_N}} \prod_{i=0}^N w_{\text{type}_{O}(i,0)} \right) \left(\sum_{i=1}^8 w_i \right) \\
& \text{(it is the decomposition according to the orientations of edges adjacent to node $(N+1,0)$)} \\
& = (a+a+b+b+c+c+d+d) Z^8_{L_N} = 2^2 (a+c) Z^8_{L_N}.
\end{align*}
where, for any $N$, $\Omega^8_{L_N}$ is the set of 8-vertex model configurations of $L_N$ and $\text{type}_{O}(i,0)$ is the type of vertex $(i,0)$ in orientation $O$ of $L_N$.

The other point to observe if that for any graph $G$, finite subgraph of $\overline{K}_\infty$, and any vertex $v=(i,j)$ of $G$ such that edges $(2i,2j)$ and $(2i+1,2j)$ are internal edges and that edges $(2i,2j+1)$ and $(2i+1,2j+1)$ are external edges, then the partition function $Z^8_G$ of 8-vertex model on $G$ with FBC satisfies
\begin{equation}
Z^{8}_{G} = (a+c) Z^{8}_{\prive{G}{v}}
\end{equation}
if $a+c=b+d$, where $Z^{8}_{\prive{G}{v}}$ is the partition function of 8-vertex model on $\prive{G}{v}$ with FBC. Indeed, let us decompose $\Omega^8_G$ (resp. $\Omega^8_{\prive{G}{v}}$) the set of 8-vertex model configurations of $G$ (resp. $\prive{G}{v}$) into four subsets $\Omega_G^{8,(e_1,e_2)}$ (resp. $\Omega_{\prive{G}{v}}^{8,(e_1,e_2)}$) with $e_1,e_2 \in \{0,1\}$ where $\Omega_G^{8,(e_1,e_2)}$ (resp. $\Omega_{\prive{G}{v}}^{8,(e_1,e_2)}$) is the set of 8-vertex model configurations of $G$ (resp. $\prive{G}{v}$) such that $e(2i,2j) = e_1$ and $e(2i+1,2j)=e_2$. Then
\begin{align*}
Z^{8}_{G} & = \sum_{O \in \Omega^8_{G}} \prod_{v' \in V_{G}} w_{\text{type}_{O}(v')} \\
& = \sum_{O \in \Omega_{G}^{8,(0,0)}} \prod_{v' \in V_{G}} w_{\text{type}_{O}(v')} + \sum_{O \in \Omega_{G}^{8,(0,1)}} \prod_{v' \in V_{G}} w_{\text{type}_{O}(v')} \nonumber \\
& \qquad + \sum_{O \in \Omega_{G}^{8,(1,0)}} \prod_{v' \in V_{G}} w_{\text{type}_{O}(v')} + \sum_{O \in \Omega_{G}^{8,(1,1)}} \prod_{v' \in V_{G}} w_{\text{type}_{O}(v')} \\
& = \sum_{O \in \Omega_{\prive{G}{v}}^{8,(0,0)}} \prod_{v' \in V_{\prive{G}{v}}} w_{\text{type}_{O}(v')}\ (b+d) + \sum_{O \in \Omega_{\prive{G}{v}}^{8,(0,1)}} \prod_{v' \in V_{\prive{G}{v}}} w_{\text{type}_{O}(v')}\ (a+c) \nonumber \\
& \qquad + \sum_{O \in \Omega_{\prive{G}{v}}^{8,(1,0)}} \prod_{v' \in V_{\prive{G}{v}}} w_{\text{type}_{O}(v')}\ (a+c) + \sum_{O \in \Omega_{\prive{G}{v}}^{8,(1,1)}} \prod_{v' \in V_{\prive{G}{v}}} w_{\text{type}_{O}(v')}\ (b+d)\\
& \text{(by decomposition according the possible orientations of edges $(2i,2j+1)$ and $(2i+1,2j+1)$)}\\
& = (a+c) \sum_{O \in \Omega^8_{\prive{G}{v}}} \prod_{v \in V_{\prive{G}{v}}} w_{\text{type}_{O}(v')} = (a+c) Z^{8}_{\prive{G}{v}}
\end{align*}
\end{proof}

Now, we can prove Proposition~\ref{prop:8bc}.
\begin{proof}[Proof of Proposition~\ref{prop:8bc}]
Let $O \in \overline{\Omega}^8_N$. Then 
\begin{equation}
\overline{P}^{8}_N(O) = \frac{a^{n_1(O)+n_2(O)} b^{n_3(O)+n_4(O)} c^{n_5(O)+n_6(O)} d^{n_7(O)+n_8(O)}}{2^{2N}(a+c)^{N(N+1)/2}} 
\end{equation}
and, by~\eqref{eq:HPMBC},
\begin{equation}
\overline{P}^{8,1/2}_N(O) = \sum_{B \in \{0,1\}^{2N}} \frac{1}{2^{2N}} \frac{a^{n_1(O)+n_2(O)} b^{n_3(O)+n_4(O)} c^{n_5(O)+n_6(O)} d^{n_7(O)+n_8(O)}}{\overline{Z}^{8,B}_N} \ind{(e(i,0): 0 \leq i \leq 2N-1)=B}.
\end{equation}
But, for any $B \in \{0,1\}^{2N}$, $\overline{Z}^{8,B}_N = (a+c)^{N(N+1)/2}$ by arguments used to prove Lemma~\ref{lem:pf8} and the following remark: for $N=1$, 
\begin{align*}
\overline{Z}^{8,B}_{1} & = 
\begin{cases} 
(b+d) & \text{ if } B \in \{(0,0),(1,1)\} \\
(a+c) & \text{ if } B \in \{(0,1),(1,0)\}
\end{cases} \\
& = a+c.
\end{align*}
And, so
\begin{align}
\overline{P}^{8,1/2}_{N}(O) & = \frac{1}{2^{2N}} \frac{a^{n_1(O)+n_2(O)} b^{n_3(O)+n_4(O)} c^{n_5(O)+n_6(O)} d^{n_7(O)+n_8(O)}}{(a+c)^{N(N+1)/2}} \\
& = \overline{P}^{8}_N(O).
\end{align}
\end{proof}

Proposition~\ref{prop:limitlaw} is a direct consequence of Lemma~\ref{lem:pf8}.
\begin{proof}[Proof of Proposition~\ref{prop:limitlaw}]
For any $N$, let $S_N$ be the set of edges
\begin{align*}
S_N =  & \{ (i,0) : -2 \lfloor N/2 \rfloor \leq i \leq 2\lceil N/2 \rceil-1 \} \\
& \bigcup_{j=0}^{N-1} \left\{ (i,j+1)~|~ -2 \lfloor N/2 \rfloor +j \leq i \leq 2\lceil N/2 \rceil-1-j  \right\}.
\end{align*}

Let $O \sim \mathcal{L}(\PM{1/2};a/(a+c),b/(b+d))$ with $a+c=b+d$. Then, orientation $O_{|S_N}$ is distributed according to the following probability, by Definition~\ref{def:loiqpr},
\begin{align*}
\prob{O_{|S_N}} & = \frac{1}{2^{2N}} \left( \frac{a}{a+c} \right)^{n_1+n_2}  \left( \frac{b}{b+d} \right)^{n_3+n_4} \left( \frac{c}{a+c} \right)^{n_5+n_6} \left( \frac{d}{b+d} \right)^{n_7+n_8} \\
& =  \frac{a^{n_1+n_2} b^{n_3+n_4} c^{n_5+n_6} d^{n_7+n_8}}{2^{2N}(a+c)^{N(N+1)/2}}\\
& = \overline{P}^{8}_N(O_{|S_N}).
\end{align*}
\end{proof}

\subsection{Partition function on $K_N$}
\begin{lemma}
If $a+c=b+d$, then for any $N$, the partition function $Z^{8}_{N}$ of the 8-vertex model on $K_N$ with $\text{FBC}$ is
\begin{equation}
Z^{8}_{N} = 2^{2N} (a+c)^{N^2}
\end{equation}
\end{lemma}

\begin{proof}
Notations used here are the same as in proof of Lemma~\ref{lem:pf8}. Moreover, many arguments are the same as Lemma~\ref{lem:pf8}. The only new argument is the following one: for any graph $G$, finite subgraph of $\overline{K}_\infty$, and any vertex $v=(i,j)$ of $G$ such that edge $(2i,2j)$ (resp. $(2i+1,2j)$) is internal and edges $(2i+1,2j)$ (resp. $(2i,2j)$), $(2i,2j+1)$ and $(2i+1,2j+1)$ are external edges, then the partition function $Z^8_G$ of 8-vertex model on $G$ with FBC satisfies
\begin{equation}
Z^{8}_{G} = 2 (a+c) Z^{8}_{\prive{G}{v}}.
\end{equation}
We treat the case where $(2i,2j)$ is the internal edge of $v$. Let us decompose the set $\Omega^8_{\prive{G}{v}}$ into two subsets $\Omega^{8,0}_{\prive{G}{v}}$ and $\Omega^{8,1}_{\prive{G}{v}}$ such that, for any $k \in \{0,1\}$, $\Omega^{8,k}_{\prive{G}{v}}$ is the set of 8-vertex model configurations of $\prive{G}{v}$ such that $e(2i,2j) = k$. Then, decomposing according to possible orientations of $v$, 
\begin{align*}
Z^{8}_{G} & = a \sum_{O \in \Omega_{\prive{G}{v}}^{8,0}} W(O) + b \sum_{O \in \Omega_{\prive{G}{v}}^{8,0}} W(O) + c \sum_{O \in \Omega_{\prive{G}{v}}^{8,0}} W(O) + d \sum_{O \in \Omega_{\prive{G}{v}}^{8,0}} W(O) \nonumber \\
& \qquad + a \sum_{O \in \Omega_{\prive{G}{v}}^{8,1}} W(O) + b \sum_{O \in \Omega_{\prive{G}{v}}^{8,1}} \prod_{v' \in V_{G}} W(O) + c \sum_{O \in \Omega_{\prive{G}{v}}^{8,1}} W(O) + d \sum_{O \in \Omega_{\prive{G}{v}}^{8,1}} \prod_{v' \in V_{G}} W(O) \\
& = (a+b+c+d) \left( \sum_{O \in \Omega_{G}^{8,0}} W(O) + \sum_{O \in \Omega_{G}^{8,1}} W(O) \right) \\
& = 2 (a+c) Z^{8}_{\prive{G}{v}}.
\end{align*}
\end{proof}

Now, we present a useful property of $\overline{P}^8_{\infty} = \mathcal{L} \left(\PM{1/2};p,r \right)$.
\begin{lemma} \label{lem:iid}
Let $O \sim \overline{P}^8_{\infty}$. Let $(t_i: i \in \ZZ) \in \NN^\ZZ$ such that $t_{i+1}-t_i \in \{0,(-1)^{i+1+t_i}\}$. 
Then $(e(i,t_i) : i \in \NN) \sim \PM{1/2}$.
\end{lemma}
Its proof is done in Section~\ref{sec:vitesse}. Indeed, the lemma is a direct consequence of Proposition~\ref{prop:part}, presented and proved in Section~\ref{sec:vitesse}.

Now, with these two lemmas, we can prove Proposition~\ref{prop:SKN}.
\begin{proof}[Proof of Proposition~\ref{prop:SKN}]
Let $O \sim \overline{P}^8_{\infty}$ and let $S_N = (V_N,E_N)$. Then, there exists $(i,t) \in \ZZ \times \NN$, $i+t$ even, such that $E_N = \{(i',t') : i+t \leq i'+ t' \leq i+t+2N, i-t - 2N +1 \leq i'-t' \leq i+1-t \}$ (see Figure~\ref{fig:SKN}). In particular, $(t_{i+j}= t+j: -N+1\leq j \leq 0) \cup (t_{i+j} = t - j +1 : 1 \leq j \leq N)$ satisfies condition of Lemma~\ref{lem:iid} and, so, $(e(i+j,t_{i+j}) : -N+1 \leq j \leq N)$ are i.i.d.\ and of law $\mathcal{B}(1/2)$. Hence,
\begin{align*}
\prob{O_{|S_N}} & = \frac{1}{2^{2N}} \left(\frac{a}{a+c}\right)^{n_1+n_2} \left(\frac{b}{b+d}\right)^{n_3+n_4} \left( \frac{c}{a+c}\right)^{n_5+n_6} \left(\frac{d}{b+d}\right)^{n_7+n_8} \\
& = \frac{a^{n_1+n_2} b^{n_3+n_4} c^{n_5+n_6} d^{n_7+n_8}}{Z^8_N}
\end{align*}
\end{proof}

\section{Exact computation of the edge correlation function} \label{sec:Cit}
To prove Theorem~\ref{thm:Cit}, we need to prove first the following proposition.
\begin{proposition} \label{prop:C8V}
The edge correlation function $C_8(i,t)$ (of 8-vertex model with FBC and $a+c=b+d$) is the coefficient of $l^t x^{i+t}$ in the formal series of the following rational fraction:
\begin{equation}
\frac{1 + l (1 - (p+r) +x (r-p))}{x^{2} l^{2} (2p-1) (2r-1) + l (1-(p+r)) (1+x^2) + 1} \label{eq:Cit}
\end{equation}
with $\displaystyle p = \frac{a}{a+c}$ and $\displaystyle r = \frac{b}{b+d}$.
\end{proposition}
After proving this proposition, we will extract coefficients of~\eqref{eq:Cit} to prove Theorem~\ref{thm:Cit}.

\subsection{Proof of Proposition~\ref{prop:C8V}} \label{sec:P8V}
Let $O \sim \overline{P}^8_\infty$. We recall that, in Definition~\ref{def:loiqpr}, there are two operators $T_0$ and $T_1$ that give orientations line by line according to parity of time.\par

In the following, we suppose that $r \leq 1-p$. Case $1-p \leq r$ can be treated in a similar way with some differences that are commented in Remark~\ref{rem:ordre}. 
First, let us compute $C_8(i,t)$:
\begin{align}
C_8(i,t) & = \frac{\esp{}{e(0,0)e(i,t)} - \esp{}{e(0,0)} \esp{}{e(i,t)}}{\sqrt{\var{e(i,t)}} \sqrt{\var{e(0,0)}}} = 4 \left(\prob{e(i,t) = 1 \text{ and } e(0,0) = 1} - \frac{1}{4}\right) \nonumber \\
& = 2 \prob{e(i,t) = 1~|~e(0,0) = 1} - 1 \nonumber \\
& = \prob{e(i,t) = 1~|~e(0,0) = 1} - \prob{e(i,t) = 0~|~e(0,0) = 1} \label{eq:COT}
\end{align}
Hence, we need to compute
\begin{displaymath}
\prob{e(i,t) = 1~|~e(0,0) = 1} \text{ or } \prob{e(i,t) = 0~|~e(0,0) = 1}.
\end{displaymath}
This is done using Definition~\ref{def:loiqpr} and the following crucial lemma. In few words, this lemma tells us that after a transition $T_0$ or $T_1$, orientation $e(i,t)$ of edge $(i,t)$ influences orientation $e(j,t+1)$ of a unique random edge on the line $((j,t+1) : j \in \ZZ)$.
\begin{lemma} \label{lem:couplage}
Let $i \in \ZZ$. Let $(s(j) : j \in \ZZ)$ be a sequence of random variables whose values are in $\{0,1\}$ such that $s(i)$ is independent of $(s(j) : j \neq i)$. For any $u \in \{0,1\}$, we denote $s_u = T_u(s)$, then
\begin{itemize}[topsep=0pt,itemsep=0pt,partopsep=0pt,parsep=0pt]
\item with probability $r$, $s_u(i) = s(i)$ and $(s_u(j):j \neq i)$ are independent of $s(i)$;
\item with probability $1-p-r$, $s_u(i+(-1)^{i+u}) = 1-s(i)$ and $(s_u(j):j \neq i+(-1)^{i+u})$ are independent of $s(i)$;
\item with probability $p$, $s_u(i) = 1-s(i)$ and $(s_u(j):j \neq i)$ are independent of $s(i)$.
\end{itemize} 
\end{lemma}

\begin{proof}
We establish the proof for $u=0$ and $i$ even, other cases are proved in similar ways.\par
By definition of $T_0$,
\begin{itemize}[topsep=0pt,itemsep=0pt,partopsep=0pt,parsep=0pt]
\item if $s(i+1) = 1- s(i)$, then
\begin{itemize}[topsep=0pt,itemsep=0pt,partopsep=0pt,parsep=0pt]
\item with probability $1-p$, $s_0(i) = s(i) = 1 - s(i+1)$ and $s_0(i+1)=s(i+1)=1-s(i)$, and
\item with probability $p$, $s_0(i) = 1-s(i)$ and $s_0(i+1) = 1-s(i+1)$; but
\end{itemize}
\item if $s(i+1) = s(i)$, then
\begin{itemize}[topsep=0pt,itemsep=0pt,partopsep=0pt,parsep=0pt]
\item with probability $r$, $s_0(i) = s(i)$ and $s_0(i+1)=s(i+1)$, and
\item with probability $1-r$, $s_0(i) = 1-s(i) = 1- s(i+1)$ and $s_0(i+1) = 1-s(i+1)=1-s(i)$.
\end{itemize}
\end{itemize}
\smallskip
Now, observe, from the fact that $r \leq 1-p$ and, by a coupling argument, that
\begin{itemize}[topsep=0pt,itemsep=0pt,partopsep=0pt,parsep=0pt]
\item with probability $r$, $s_0(i)=s(i)$ and $s_0(i)$ depends only on $s(i)$ and not on $s(i+1)$;
\item with probability $1-p-r$, $s_0(i+1)=1-s(i)$ and $s_0(i+1)$ depends only on $s(i)$ and not on $s(i+1)$;
\item with probability $p$, $s_0(i)=1-s(i)$ and $s_0(i)$ depends only on $s(i)$ and not on $s(i+1)$.
\end{itemize}
Table~\ref{tab:RecC} shows this coupling argument.

\begin{table}[h!] 
\begin{center}
\begin{tabular}{|l|c|c||c|c|}
\hline
probability & $s_0(i)$ & $s_0(i+1)$ & \raisebox{-6mm}{\includegraphics{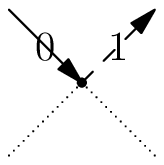}} & \raisebox{-6mm}{\includegraphics{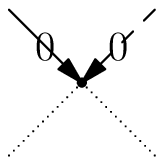}} \\
\hline
$r$ & $s(i)$ & $s(i+1)$ & \raisebox{-6mm}{\includegraphics{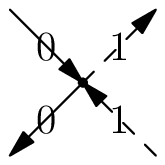}} & \raisebox{-6mm}{\includegraphics{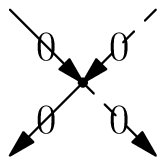}} \\
\hline
$1-p-r$ & $1-s(i+1)$ & $1-s(i)$ & \raisebox{-6mm}{\includegraphics{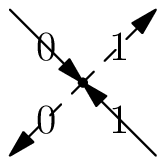}} & \raisebox{-6mm}{\includegraphics{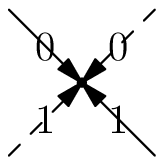}} \\
\hline
$p$ & $1-s(i)$ & $1-s(i+1)$ & \raisebox{-6mm}{\includegraphics{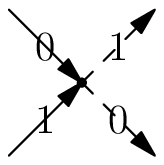}} & \raisebox{-6mm}{\includegraphics{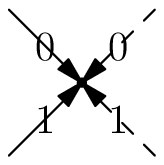}} \\
\hline
\end{tabular}
\end{center}
\caption{Coupling in the proof of Lemma~\ref{lem:couplage}. Figures on the fourth and fifth columns represents the coupling when $s(i)=1-s(i+1)=0$ in the fourth and when $s(i)=s(i+1)=0$ in the fifth. 
}
\label{tab:RecC}
\end{table}

Finally, by definition of $T_0$, $(s_0(j) : j \neq i,i+1)$ depend only on $(s(j): j \neq i,i+1)$ and not on $s(i)$. That is ending the proof.
\end{proof}

Lemma~\ref{lem:couplage} is crucial and central. Indeed, it permits itself to understand exact influences of orientation $e(0,0)$ on an orientation $O \sim \overline{P}^8_\infty$. This influence could be rewritten in term of a non-homogeneous random walk.
\begin{definition} \label{def:XMA}
Let $p,r \in [0,1]$, $i_0 \in \ZZ$ and $k_0 \in \{0,1\}$. We denote by $(X_t: t \geq 0)$ the following stochastic process with value on $\ZZ \times \{0,1\}$:
\begin{itemize}[topsep=0pt,itemsep=0pt,partopsep=0pt,parsep=0pt]
\item $X_0 = (i_0,k_0)$ a.s.;
\item if $X_t = (i,k)$, then 
\begin{equation} \label{eq:XMA}
X_{t+1} = \begin{cases}
(i,k) & \text{w.p. } r, \\
(i+(-1)^{i+t},1-k) & \text{w.p. } 1-p-r, \\
(i,1-k) & \text{w.p. } p.
\end{cases}
\end{equation}
\end{itemize}
This stochastic process is a Markov chain, but not a homogeneous Markov chain because its transitions depend on time's parity.
\end{definition}

\begin{lemma}
\begin{equation}
C_8(i,t) = \prob{X_t = (i,1) | X_0 = (0,1)} - \prob{X_t=(i,0) | X_0 = (0,1)}. \label{eq:COT2}
\end{equation}
\end{lemma}

\begin{proof}
By~\eqref{eq:COT}, it is sufficient to prove that if $O \sim \overline{P}^8_\infty$, then for, any $i \in \ZZ$, any $t \in \NN$,
\begin{align}
& \prob{e(i,t) = 1~|~e(0,0) = 1} -\prob{e(i,t) = 0~|~e(0,0) = 1} \nonumber \\
& = \prob{X_t = (i,1)~|~X_0 = (0,1)} - \prob{X_t = (i,0)~|~X_0 = (0,1)}. \label{eq:CXX}
\end{align}

This is done by induction on $t$. When $t=0$, we have that, for any $i$, for any $k$,
\begin{align*}
& \prob{e(i,0) = k~|~e(0,0) = 1} -\prob{e(i,0) = 0~|~e(0,0) = 1} \\
& = \begin{cases} 
1-0  & \text{if } i=0\\
1/2-1/2 & \text{else}
\end{cases}\\
& = \prob{X_0 = (i,1) | X_0 = (0,1)}-\prob{X_0 = (i,0) | X_0 = (0,1)} .
\end{align*}
Now, let $t \in \NN$ and suppose that~\eqref{eq:CXX} holds for any $i$, then
\begin{align*}
& \prob{e(i,t+1) = 1~|~e(0,0) = 1} -\prob{e(i,t+1) = 0~|~e(0,0) = 1} \nonumber \\
& = r\, \prob{e(i,t) = 1~|~e(0,0) = 1} - r\, \prob{e(i,t) = 0~|~e(0,0) = 1} \\
& \quad + (1-p-r) \prob{e(i+(-1)^{i+t},t) = 0~|~e(0,0) = 1} - (1-p-r) \prob{e(i+(-1)^{i+t},t) = 1~|~e(0,0) = 1} \\
& \quad + p\, \prob{e(i,t) = 0~|~e(0,0) = 1} - p\, \prob{e(i,t) = 1~|~e(0,0) = 1} \text{ (by Lemma~\ref{lem:couplage})}\\
& = r (\prob{X_t = (i,1) | X_0 = (0,1)}-\prob{X_{t+1} = (i,0) | X_0 = (0,1)})\\
& \quad + (1-p-r) (\prob{X_t = (i+(-1)^{i+t},0) | X_0 = (0,1)}-\prob{X_{t+1} = (i+(-1)^{i+t},1) | X_0 = (0,1)})\\
& \quad + p (\prob{X_t = (i,0) | X_0 = (0,1)}-\prob{X_{t+1} = (i,1) | X_0 = (0,1)})\\
& = \prob{X_{t+1} = (i,1) | X_0 = (0,1)}-\prob{X_{t+1} = (i,0) | X_0 = (0,1)} \text{ (by Definition~\eqref{eq:XMA})}
\end{align*}
\end{proof}

Hence, computations of probabilities that $X_t$ is in a certain state is equivalent to compute $C_8$. To do that, we use generating functions and methods of analytic combinatorics. For references to these methods, we recommend the book of Flajolet-Sedgewick~\cite{FS09}.\par
Let $\tilde{E}$ be the set of paths that start from $(0,0)$ and go to any point $(i,t) \in \ZZ \times \NN$ using only three steps $(-1,1)$, $(0,1)$ or $(1,1)$ and such that, if the path is in a node $(i',t')$, then the next allowed steps are $((-1)^{i'+t'},1)$ and $(0,1)$. In other words, $\tilde{E}$ is the set of paths of the graph represented on Figure~\ref{fig:path} with starting point $(0,0)$.
\begin{figure}
\begin{center}
\includegraphics{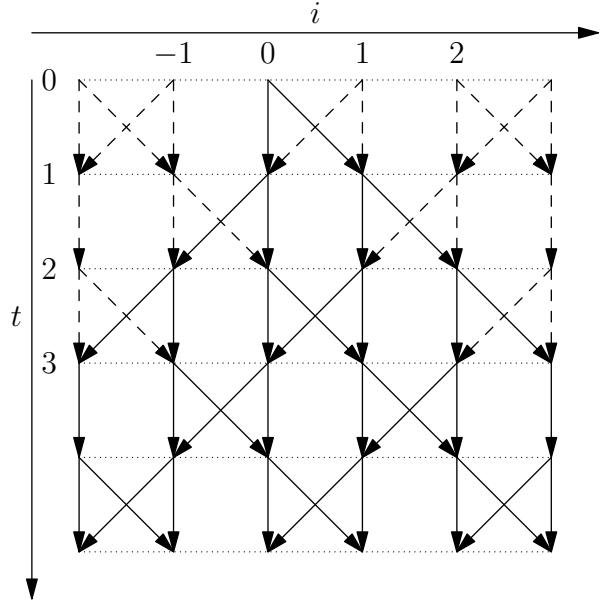}
\end{center}
\caption{The directed graph on which $\tilde{E}$ is supported. In full line, the first possible steps of paths in $\tilde{E}$.}
\label{fig:path}
\end{figure}

A colored path of $\tilde{E}$ is a pair $(w,s)$ where $w \in \tilde{E}$ and $s$ is a function from $w$ to $\{0,1\}$ (for any $u\in w$, $s(u)$ is called the color of $u$). We denote by $E$ the set of colored paths $(w = ((0,0),(i_1,1),\dots,(i_{t-1},t-1),(i,t)),s)$ of $\tilde{E}$ that satisfy the following constraints:
\begin{itemize}[topsep=0pt,itemsep=0pt,partopsep=0pt,parsep=0pt]
\item $s((0,0))=1$ and
\item for any $1 \leq j \leq t$, $s((i_j,j)) = 1- s((i_{j-1},j-1))$ if $|i_j-i_{j-1}| = 1$, in other words: if the step from $(i_{j-1},j-1)$ to $(i_j,j)$ is diagonal ($(-1,1)$ or $(1,1)$), then the color changes.
\end{itemize}

To be able to count elements of $E$ according to statistics defined later, we decompose $E$ into four subsets that form a partition of $E$: for any $k_1,k_2 \in \{0,1\}$, $E_{k_1,k_2}$ is the subset of $E$ of colored paths that finish in a node $(i,t)$ such that $i+t = k_1 \!\! \mod 2$ and $s(i,t) = k_2$.

In Flajolet-Sedgewick's symbolism, relations between these sets are
\begin{align}
E_{0,1} = 1 + \overset{E_{0,0}}{\phantom{|}} \searrow \underset{1}{\phantom{|}} + \underset{1}{\overset{E_{1,1}}{\downarrow}} + \underset{1}{\overset{E_{1,0}}{\downarrow}},\ 
E_{0,0} = \overset{E_{0,1}}{\phantom{|}} \searrow \underset{0}{\phantom{|}} + \underset{0}{\overset{E_{1,0}}{\downarrow}} + \underset{0}{\overset{E_{1,1}}{\downarrow}}, \label{eq:rec1} \\ 
E_{1,1} = \underset{1}{\phantom{|}} \swarrow \overset{E_{1,0}}{\phantom{|}} + \underset{1}{\overset{E_{0,1}}{\downarrow}} + \underset{1}{\overset{E_{0,0}}{\downarrow}},\
E_{1,0} = \underset{0}{\phantom{|}} \swarrow \overset{E_{1,1}}{\phantom{|}} + \underset{0}{\overset{E_{0,0}}{\downarrow}} + \underset{0}{\overset{E_{0,1}}{\downarrow}}. \label{eq:rec2}
\end{align}

Now, we enumerate these four subsets according to six statistics:
\begin{itemize}[topsep=0pt,itemsep=0pt,partopsep=0pt,parsep=0pt]
\item $n_v(w)$, number of vertical edges ($(0,1)$) in colored path $(w,s)$,
\item $n_d(w)$, number of diagonal edges ($(-1,1)$ or $(1,1)$) in colored path $(w,s)$,
\item $t(w)$, number of edges in colored path $(w,s)$,
\item $i(w) = i+t$ where $(i,t)$ is the final node of colored path $(w,s)$,
\item $n_c(w,s)$, number of color changes that occur on a vertical edge in colored path $(w,s)$ and
\item $n_k(w,s)$, number of no color changes that occur on a vertical edge in colored path $(w,s)$.
\end{itemize}
Some of these statistics are redundant, e.g.\ $n_c+n_k=n_v$ or $n_v+n_d=t$. 

To enumerate subsets $E_{k_1,k_2}$ according to these six statistics, we define generating functions by, for any $k_1,k_2 \in \{0,1\}$,
\begin{equation}
F_{k_1,k_2}(z_v,z_d,l,x,z_c,z_k) = \sum_{(w,s) \in E_{k_1,k_2}} z_v^{n_v(w)} z_d^{n_d(w)} l^{t(w)} x^{i(w)} z_c^{n_c(w,s)} z_k^{n_k(w,s)}.
\end{equation}

We can remark that if we take $z_d = (1-(p+r))$, $z_v = p+r$, $z_k=r/(p+r)$ and $z_c = p/(p+r)$, then, for any $k_1,k_2 \in \{0,1\}$,
\begin{align}
F_{k_1,k_2} = & \sum_{t \in \NN, i \in \ZZ | i+t = k_1 \!\!\!\!\! \mod 2} \prob{X_t = (i,k_2) | X_0 = (0,1)} l^t x^{i+t}. \label{eq:prob}
\end{align}

Hence, $C_8(i,t)$ is the coefficient of $l^t x^{i+t}$ in series development of $F_{0,1}-F_{0,0}+F_{1,1}-F_{1,0}$ evaluated in $z_d = (1-(p+r))$, $z_v = p+r$, $z_k=r/(p+r)$ and $z_c = p/(p+r)$.\par

To compute these generating functions, we use equations~\eqref{eq:rec1} and~\eqref{eq:rec2} that are translated at level of generating functions into
\begin{align}
F_{0,1} & = 1 + R\ F_{0,0} + K\ F_{1,1} + C\ F_{1,0}, \nonumber \\
F_{0,0} & = R\ F_{0,1} + K\ F_{1,0} + C\ F_{1,1}, \nonumber \\
F_{1,1} & = L\ F_{1,0} + K\ F_{0,1} + C\ F_{0,0}, \nonumber \\
F_{1,0} & = L\ F_{1,1} + K\ F_{0,0} + C\ F_{0,1}. \nonumber
\end{align}
with $R=z_dlx^2$, $L=z_dl$, $K = z_kz_vlx$ and $C=z_cz_vlx$.

This system is solved (by hand or with help of a formal computation software as Sage) and its resolution gives:
\begin{align}
F_{0,1} = & \frac{1- 2 C R K - C^{2} - R^{2} - K^{2}}{H} \\
F_{0,0} = & \frac{C^{2} R - L R^{2} + R K^{2} + 2 C K + L }{H} \\
F_{1,1} = & \frac{-K^3 + C (R+L) + K (C^2 + RL + 1)}{H} \\
F_{1,0} = & \frac{-C^3 + K (R+L) + C (K^2 + RL + 1)}{H}
\end{align}
with $H = \left((C+K)^{2} - (1-R)(1-L)\right) \left((C-K)^{2} - (1+R)(1+L)\right)$. And, so, 
\begin{equation}
F_{0,1}-F_{0,0}+F_{1,1}-F_{1,0} = \frac{C-R-K-1}{(C-K)^{2} - (1+R)(1+L)}
\end{equation}
that, evaluated in $C=plx$, $K=rlx$, $R=(1-p-r)l$ and $L=(1-p-r)lx^2$, gives the rational fraction~\eqref{eq:Cit}. That ends the proof of Proposition~\ref{prop:C8V} in the case $r \leq 1-p$.\par

\bigskip

\begin{remark} \label{rem:ordre}
In the case $1-p \leq r$, Lemma~\ref{lem:couplage} is changed by the following lemma (changes between the two lemmas are indicated in bold):\par
\begin{lemma} \label{lem:couplage2}
Let $i \in \ZZ$. Let $(s(j) : j \in \ZZ)$ a sequence of random variables whose values are in $\{0,1\}$ such that $s(i)$ is independent of $(s(j) : j \neq i)$. For any $u \in \{0,1\}$, we denote $s_u=T_u(s)$, then
\begin{itemize}[topsep=0pt,itemsep=0pt,partopsep=0pt,parsep=0pt]
\item with probability $\red{1-p}$, $s_u(i) = s(i)$ and $(s_u(j):j \neq i)$ are independent of $s(i)$;
\item with probability $\red{r+p-1}$, $s_u(i+(-1)^{i+u}) = \red{s(i)}$ and $(s_u(j):j \neq i+(-1)^i)$ are independent of $s(i)$;
\item with probability $\red{1-r}$, $s_u(i) = 1-s(i)$ and $(s_u(j):j \neq i)$ are independent of $s(i)$.
\end{itemize} 
\end{lemma}

Proof of this lemma is similar to the one of Lemma~\ref{lem:couplage}. Table~\ref{tab:RecC} becomes Table~\ref{tab:RecC2}. 

Then, these changes impact proof as following. First, we need to change Definition~\ref{def:XMA} of the random walk $X$ on $\ZZ \times \{0,1\}$ in consequence. Then, we enumerate set $E'$ of colored paths $(w,s)$ that have the following constraints: $s(0,0) = 1$ and $s(i_j,j) = s(i_{j-1},j)$ when $|i_j-i_{j-1}|=1$. We separate them as before in four subsets of colored paths whose generation functions $F'$ can be computed. We evaluate $F'_{0,1}-F'_{0,0}+F'_{1,1}-F'_{1,0}$ in $z_d = (p+r-1)$, $z_v = 2-(p+r)$, $z_k=(1-p)/(2-(p+r))$ and $z_c = (1-r)/(2-(p+r))$ to finally obtain the same rational fraction~\eqref{eq:Cit}.
\end{remark}

\begin{table}[h!] 
\begin{center}
\begin{tabular}{|l|c|c||c|c|}
\hline
probability & $s_0(i)$ & $s_0(i+1)$ & \raisebox{-6mm}{\includegraphics{typec1}} & \raisebox{-6mm}{\includegraphics{typec2}} \\
\hline
$1-p$ & $s(i)$ & $s(i+1)$ & \raisebox{-6mm}{\includegraphics{typec1p}} & \raisebox{-6mm}{\includegraphics{typec2p}} \\
\hline
$r+p-1$ & $s(i+1)$ & $s(i)$ & \raisebox{-6mm}{\includegraphics{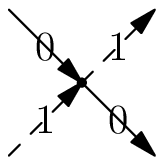}} & \raisebox{-6mm}{\includegraphics{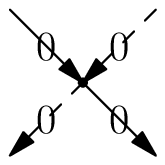}} \\
\hline
$1-r$ & $1-s(i)$ & $1-s(i+1)$ & \raisebox{-6mm}{\includegraphics{typec1r}} & \raisebox{-6mm}{\includegraphics{typec2r}} \\
\hline
\end{tabular}
\end{center}
\caption{Coupling in the Lemma~\ref{lem:couplage2}. Figures on the fourth and fifth columns represents the coupling when $s(i)=1-s(i+1)=0$ in the fourth and when $s(i)=s(i+1)=0$ in the fifth. 
}
\label{tab:RecC2}
\end{table}

\subsection{Proof of Theorem~\ref{thm:Cit}} \label{sec:extract}
To prove Theorem~\ref{thm:Cit}, we develop in formal series the rational fraction~\eqref{eq:Cit} according to $l$ and $x$ and we extract coefficients of $l^t x^{i+t}$ to get $C_8(i,t)$. With notations of~\eqref{eq:DDP}, the rational fraction~\eqref{eq:Cit} is
\begin{equation}
\frac{1 + \left(\Delta +x D\right) l}{1 + \Delta (1+x^2) l + P x^{2} l^{2}} \label{eq:Cit2}
\end{equation}

Let us begin the development of this rational fraction according to $l$.  
\begin{lemma} 
Take $t \geq 0$. The coefficient of $l^t$ in formal series of the rational fraction~\eqref{eq:Cit2} is
\begin{equation}
f(t) + (\Delta + xD) f(t-1)
\end{equation}
with, for any $t \geq 0$, 
\begin{equation}
f(t) = \sum_{k=0}^{\lfloor t/2 \rfloor} \binom{t-k}{k} (-1)^{t-k} \Delta^{t-2k} P^k x^{2k} (1+x^2)^{t-2k} \text{ and } f(-1) = 0.
\end{equation}
And, so, for any $t \geq 1$,
\begin{align}
& f(t) +  (\Delta + xD) f(t-1) = \nonumber \\
& \qquad (-1)^{t/2} P^{t/2} x^{t} \ind{t = 0 \text{ mod } 2} \nonumber \\
& \qquad  + \sum_{k=0}^{\lfloor \frac{t-1}{2} \rfloor} (-1)^{t+k} \left(\binom{t-1-k}{k-1} \Delta - \binom{t-1-k}{k} D x + \binom{t-k}{k} \Delta x^2 \right)  \Delta^{t-1-2k} P^k x^{2k} (1+x^2)^{t-1-2k}
\label{eq:lt}
\end{align}
with convention that $\binom{n}{-1} = 0$ for any $n \in \NN$. 
\end{lemma}

\begin{proof}
We develop \eqref{eq:Cit2} in formal series according to $l$. 
\begin{align*}
& \frac{1 + l \left(\Delta +x D \right)}{x^{2} l^{2} P + l \Delta (1+x^2) + 1} \\
& =  \left(1 + l \left(\Delta +x D \right) \right) \left( \sum_{i=0}^\infty (- x^{2} l^{2} P - l \Delta (1+x^2))^i \right) \\
& =  \left(1 + l \left(\Delta +x D \right) \right) \left( \sum_{i=0}^\infty (-1)^i l^i (P x^{2} l + \Delta (1+x^2))^i \right) \\
& =  \left(1 + l \left(\Delta +x D \right) \right) \left( \sum_{i=0}^\infty (-1)^i l^i \sum_{k=0}^i \binom{i}{k} (P x^{2} l)^k  (\Delta (1+x^2))^{i-k} \right) \\
& =  \left(1 + l \left(\Delta +x D \right) \right) \left( \sum_{i=0}^\infty \sum_{k=0}^i \binom{i}{k} (-1)^i \Delta^{i-k} l^{i+k} P^k x^{2k} (1+x^2)^{i-k} \right) \\
& = \left(1 + l \left(\Delta +x D \right) \right) \left( \sum_{j=0}^\infty l^j \sum_{k=0}^{\lfloor j/2 \rfloor} \binom{j-k}{k} (-1)^{j-k} \Delta^{j-2k} P^k x^{2k} (1+x^2)^{j-2k} \right) \\
& = \sum_{j=0}^{\infty} l^j (f(j) + (\Delta + xD) f(j-1)).
\end{align*}
\end{proof}

For any $t \geq 0$, $f(t) + (\Delta + xD) f(t-1)$ is polynomial according to $x$. Now, we extract coefficients of $x^j$ for any $j \geq 0$.
\begin{lemma} \label{lem:coeffjt}
The coefficient of $x^j$ in~\eqref{eq:lt} is,
\begin{itemize}[topsep=0pt,itemsep=0pt,partopsep=0pt,parsep=0pt]
\item if $j$ is odd (we denote $j' = \floor{j/2}$),
\begin{equation}
(-1)^{t+1} D \sum_{k=0}^{\min(j',t-1-j')} (-1)^k \binom{t-1-k}{k,j'-k,t-1-j'-k} \Delta^{t-1-2k} P^k; \label{eq:ji}
\end{equation}

\item if $j=t$ is even,
\begin{equation}
\sum_{k=0}^{t/2} (-1)^{k} \binom{t-1-k}{(t/2)-k} \binom{t/2}{k} \Delta^{t-2k} P^k; \label{eq:tt}
\end{equation}

\item if $j$ is even (we denote $j'=j/2$) and $j \neq t$, 
\begin{align}
& \sum_{k=0}^{\min(j'-1,t-1-j')} (-1)^{t+k} \binom{t-1-k}{t-j'-k} \binom{j'}{k} \Delta^{t-2k} P^k \nonumber \\
& + \ind{2j'<t} (-1)^{t-j'} \binom{t-1-j'}{j'-1} \Delta^{t-2j'} P^{j'}  + \ind{2j'>t} (-1)^{j'} \binom{j'}{t-j'} \Delta^{2j'-t} P^{t-j'}. \label{eq:jt}
\end{align}
\end{itemize}
\end{lemma}

\begin{proof}
We develop the sum in~\eqref{eq:lt} according to $x$.
\begin{align*}
& \sum_{k=0}^{\lfloor \frac{t-1}{2} \rfloor} \left(\binom{t-1-k}{k-1} \Delta - \binom{t-1-k}{k} D x + \binom{t-k}{k} \Delta x^2 \right) (-1)^{t+k} \Delta^{t-1-2k} P^k x^{2k} (1+x^2)^{t-1-2k}\\
& = \sum_{k=0}^{\lfloor \frac{t-1}{2} \rfloor} \left(\binom{t-1-k}{k-1} \Delta - \binom{t-1-k}{k} D x + \binom{t-k}{k} \Delta x^2 \right) (-1)^{t+k} \Delta^{t-1-2k} P^k x^{2k} \left(\sum_{j=0}^{t-1-2k} \binom{t-1-2k}{j} x^{2j} \right) \\
& = \sum_{k=0}^{\lfloor \frac{t-1}{2} \rfloor} \sum_{j=0}^{t-1-2k} \binom{t-1-2k}{j} \left( \binom{t-1-k}{k-1} \Delta - \binom{t-1-k}{k} D x + \binom{t-k}{k} \Delta x^2 \right) (-1)^{t+k} \Delta^{t-1-2k} P^k x^{2(k+j)}.
\end{align*}

First, we make the change of variable $j'=k+j$, 
\begin{displaymath}
\sum_{k=0}^{\lfloor \frac{t-1}{2} \rfloor} \sum_{j'=k}^{t-1-k} \binom{t-1-2k}{j'-k} \left( \binom{t-1-k}{k-1} \Delta - \binom{t-1-k}{k} D x + \binom{t-k}{k} \Delta x^2 \right) (-1)^{t+k} \Delta^{t-1-2k} P^k x^{2j'}
\end{displaymath}
then, permuting the sums,
\begin{displaymath}
\sum_{j'=0}^{t-1} x^{2j'}\sum_{k=0}^{\min(j',t-1-j')} \binom{t-1-2k}{j'-k} \left( \binom{t-1-k}{k-1} \Delta - \binom{t-1-k}{k} D x + \binom{t-k}{k} \Delta x^2 \right) (-1)^{t+k} \Delta^{t-1-2k} P^k.
\end{displaymath}

The coefficient of $x^{2j'+1}$ is then
\begin{displaymath}
- D \sum_{k=0}^{\min(j',t-1-j')} \binom{t-1-2k}{j'-k} \binom{t-1-k}{k} (-1)^{t+k} \Delta^{t-1-2k} P^k
\end{displaymath}
that is~\eqref{eq:ji}.

The coefficient of $x^{2j'}$ is
\begin{displaymath}
\sum_{k=0}^{\min(j',t-1-j')} \binom{t-1-2k}{j'-k} \binom{t-1-k}{k-1} (-1)^{t+k} \Delta^{t-2k} P^k + \sum_{k=0}^{\min(j'-1,t-j')} \binom{t-1-2k}{j'-1-k} \binom{t-k}{k} (-1)^{t+k} \Delta^{t-2k} P^k
\end{displaymath}
that is, if $2j' = t$,
\begin{align*}
\sum_{k=0}^{t/2-1} (-1)^{t+k} \left(\binom{t-1-2k}{(t/2)-k} \binom{t-1-k}{k-1} + \binom{t-1-2k}{(t/2)-1-k} \binom{t-k}{k}\right) \Delta^{t-2k} P^k
\end{align*}
and adding $(-1)^{t/2} P^{t/2} x^{t} \ind{t = 0 \text{ mod } 2}$, we get~\eqref{eq:tt}.

And, in the case $2j' \neq t$, we obtain,
\begin{align*}
& \sum_{k=0}^{\min(j'-1,t-1-j')} (-1)^{t+k} \left(\binom{t-1-2k}{j'-k} \binom{t-1-k}{k-1} + \binom{t-1-2k}{j'-1-k} \binom{t-k}{k}\right) \Delta^{t-2k} P^k\\
& + \ind{\min(j',t-j')=j'} (-1)^{t-j'} \binom{t-1-j'}{j'-1} \Delta^{t-2j'} P^{j'} + \ind{\min(j',t-j')=t-j'} (-1)^{j'} \binom{j'}{t-j'} \Delta^{2j'-t} P^{t-j'},
\end{align*}
equivalent to
\begin{align*}
& \sum_{k=0}^{\min(j'-1,t-1-j')} (-1)^{t+k} \binom{t-1-k}{t-j'-k} \binom{j'}{k} \Delta^{t-2k} P^k\\
& + \ind{\min(j',t-j')=j'} (-1)^{t-j'} \binom{t-1-j'}{j'-1} \Delta^{t-2j'} P^{j'} + \ind{\min(j',t-j')=t-j'} (-1)^{j'} \binom{j'}{t-j'} \Delta^{2j'-t} P^{t-j'}
\end{align*}
that is~\eqref{eq:jt}

NB: The fact that $\displaystyle \left(\binom{t-1-2k}{j'-k} \binom{t-1-k}{k-1} + \binom{t-1-2k}{j'-1-k} \binom{t-k}{k}\right) = \binom{t-1-k}{t-j'-k} \binom{j'}{k}$ for any $t,k,j'$ could be proved using factorial notation of binomials and it is trivial if $k = 0$.
\end{proof}

\begin{proof}[Proof of Theorem~\ref{thm:Cit}]
We recall that $C_8(i,t)$ is the coefficient of $x^{i+t} l^t$ in the rational fraction~\eqref{eq:Cit} (see Proposition~\ref{prop:C8V}). These coefficients are given by Lemma~\ref{lem:coeffjt}. To conclude, we change variables from $(j=i+t,t)$ to $(i=j-t,t)$.
\end{proof}

\subsection{Particle system and proof of Proposition~\ref{prop:vitesse}} \label{sec:vitesse}
In this section, we suppose that $p+r \leq 1$. The case $p+r \geq 1$ could be treated in a similar way as explained in Remark~\ref{rem:ordre}.

First, we define a particle system that is related to the 8-vertex model. 
\begin{definition}[Particle system $\mathcal{P}(\mu;p,r)$] \label{def:part}
Let $\mu$ be a probability measure on $\{0,1\}^\ZZ$ and $p,r \in [0,1]$. The law $\mathcal{P}(\mu;p,r)$ is the following law on the set $(\ZZ \times \{0,1\})^{\ZZ \times \NN}$. At time $t=0$, for any $i \in \ZZ$, there is exactly one particle (named) $\prt{i}$ in position $p(\prt{i},0) = i$ and in a random state $s(\prt{i},0) \in \{0,1\}$ and $(s(\prt{i},0) : i \in \ZZ) \sim \mu$. Then, from time $t$ to time $t+1$, for any $i \in \ZZ$ such that $i+t$ is even, particles $\alpha$ and $\beta$ such that $p(\alpha,t) = i$ and $p(\beta,t) = i+1$ interact in the following way:
\begin{itemize}
\item with probability $p$, $p(\alpha,t+1) = i$, $s(\alpha,t+1) = s(\alpha,t)$, $p(\beta,t+1) = i+1$ and $s(\beta,t+1) = s(\beta,t)$;
\item with probability $1-p-r$, $p(\alpha,t+1) = i+1$, $s(\alpha,t+1) = 1-s(\alpha,t)$, $p(\beta,t+1) = i$ and $s(\beta,t+1) = 1-s(\beta,t)$.
\item with probability $r$, $p(\alpha,t+1) = i$, $s(\alpha,t+1) = 1-s(\alpha,t)$, $p(\beta,t+1) = i+1$ and $s(\beta,t+1) = 1-s(\beta,t)$.
\end{itemize}
To see a representation of these transitions, see fourth and fifth columns on Table~\ref{tab:RecC}: plain line represents the particle $\alpha$ and dashed line the particle $\beta$.
Moreover, all these transitions are independent. The law $\mathcal{P}(\mu;p,r)$ is then the law of the random variable $((p(\prt{i},t),s(\prt{i},t)) : i \in \ZZ, t \in \NN)$.

By definition, there is exactly one particle $\alpha$ at each time $t$ in position $i$; this particle will be denoted $\alpha(i,t)$ (and, simply, $\alpha_i$ if $t=0$).
\end{definition}

This particles system is related to 8-vertex model via the following proposition.
\begin{proposition} \label{prop:part}
If $( (p(\prt{i},t),s(\prt{i},t)) : i \in \ZZ, t \in \NN) \sim \mathcal{P}(\mu;p,r)$, then $(s(\alpha(i,t),t) : i \in \ZZ, t \in \NN) \sim \mathcal{L}(\mu;p,r)$.
\end{proposition}

\begin{proof}
This is a consequence of the coupling defined in Lemma~\ref{lem:couplage}.
\end{proof}

A first consequence of this proposition is the Lemma~\ref{lem:iid}.
\begin{proof}[Proof of Lemma~\ref{lem:iid}]
First, we begin by the following remark: if at time $t=0$, $\{s(\prt{i},0) : i \in \ZZ\}$ are independent, then for any $i \in \ZZ$, $\{s(\prt{i},t) : t \in \NN\}$ is independent of $\{s(\prt{j},t) : j \in \prive{\ZZ}{\{i\}}, t \in \NN\}$ (by Lemma~\ref{lem:couplage}).\par

Now, take any sequence $(t_i : i \in \ZZ)$ such that $t_{i+1}-t_i \in \{0,(-1)^{i+1+t_i}\}$. We know that the set of allowed positions for the particle $\alpha(i,t_i)$ is, if $i+t_i$ is even,
\begin{align*}
\{(j,t) : \prob{p(\alpha(i,t_i),t)=j} > 0\} & = \{(j,t) : i-t_i\leq j-t , j+t \leq i+t_i-1, t < t_i \} \cup \{(i,t_i)\}\\
& \quad \cup \{(j,t) : j-t \leq i-t_i , i+t_i +1 \leq j+t, t_i < t\};
\end{align*}
if $i+t_i$ is odd,
\begin{align*}
\{(j,t) : \prob{p(\alpha(i,t_i),t)=j} > 0\} & = \{(j,t) : i-t_i+1\leq j-t , j+t \leq i+t_i, t < t_i \} \cup \{(i,t_i)\}\\
& \quad \cup \{(j,t) : j-t \leq i-t_i - 1 , i+t_i \leq j+t, t_i < t\}.
\end{align*}
that intersects the set $\{(j,t_j) : j \in \ZZ\}$ in only one point that is $(i,t_i)$, see Figure~\ref{fig:mouv}. Hence, particle $\alpha(i,t_i)$ cannot be in any position $(j,t_j)$ for any $j \neq i$ and, so, for any $i,j \in \ZZ$, $i \neq j$, $\alpha(i,t_i) \neq \alpha(j,t_j)$. 

To conclude, as $(s(\prt{i},0) : \prt{i} \in \ZZ) \sim PM(1/2)$ (because $O \sim \overline{P}^8_\infty$), we get that $(e(i,t_i) = s(\alpha(i,t_i),t_i) : i \in \ZZ)$ are independent, see above. And, by Proposition~\ref{cor:inv2}, for any $i \in \ZZ$, $\prob{e(i,t_i) = 0} = \prob{e(i,t_i)=1} = 1/2$.

\begin{figure}
\begin{center}
\includegraphics{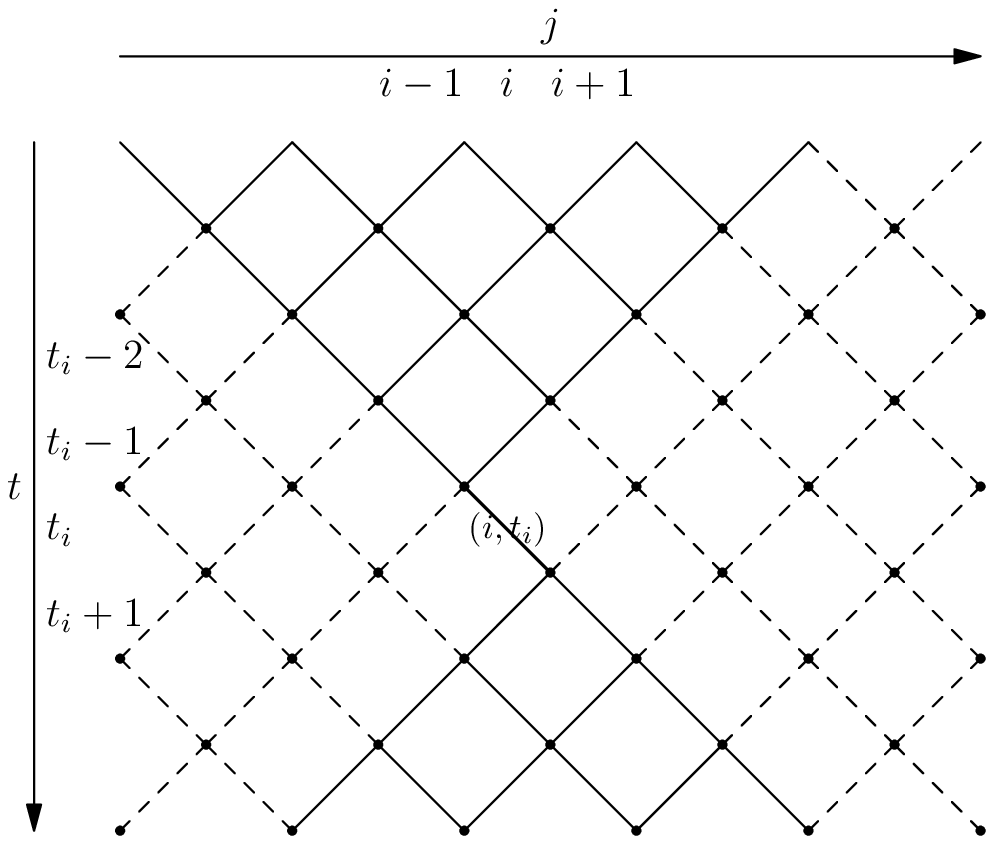}
\end{center}
\caption{In bold: the position $(i,t_i)$. In plain line: the set of allowed positions for the particle $\alpha(i,t_i)$. In dashed line: available edges for a set $((j,t_j):j \in \ZZ)$ that respects the condition $t_{i+1}-t_i \in \{0,(-1)^{i+1+t_i}\}$. There is no common position for plain and dashed line except in $(i,t_i)$.} \label{fig:mouv}
\end{figure}
\end{proof} 

A second consequence of Proposition~\ref{prop:part} is Proposition~\ref{prop:vitesse}.
Before to prove it, we need a first lemma that simulates the same system of particles, but with a change in the transition kernel.
\begin{lemma}
Let $p,r \in (0,1)$. Let $B = \{B_{i,t} : i \in \ZZ, t \in \NN\}$ (resp. $Z = \{Z_{i,t}: i \in \ZZ, t \in \NN\}$) two sets of i.i.d.\ variables of common law $\bern{1/2}$ (resp. $\bern{2m}$ with $m= \min (p,r)$).

At time $t=0$, for any $i \in \ZZ$, there is exactly one particle (named) $\prt{i}$ in position $p(\prt{i},0) = i$ and in a random state $s(\prt{i},0) \in \{0,1\}$ and $(s(\prt{i},0) : \prt{i} \in \ZZ) \sim \mu$ ($B$, $Z$ and $(s(\alpha_i,0) : i \in \ZZ)$ are mutually independent). Then, from time $t$ to time $t+1$, for any $i \in \ZZ$ such that $i+t$ is even, particles $\alpha$ and $\beta$ such that $p(\alpha,t) = i$ and $p(\beta,t) = i+1$ interact in the following way:
\begin{enumerate}
\item with probability $2m$ (that is when $Z_{i,t} =1$), $p(\alpha,t+1) = i$, $p(\beta,t+1) = i+1$, $s(\alpha,t+1) = B_{i,t}$ and  $s(\beta,t+1) = \begin{cases} 
B_{i,t} & \text{if } s(\alpha,t) = s(\beta,t) \\
1-B_{i,t} & \text{if } s(\alpha,t) = 1-s(\beta,t)
\end{cases}$;
\item with probability $p-m$ (a $Z_{i,t} =0$ case), $p(\alpha,t+1) = i$, $s(\alpha,t+1) = s(\alpha,t)$, $p(\beta,t+1) = i+1$ and $s(\beta,t+1) = s(\beta,t)$;
\item with probability $1-p-r$ (a $Z_{i,t} =0$ case), $p(\alpha,t+1) = i+1$, $s(\alpha,t+1) = 1-s(\alpha,t)$, $p(\beta,t+1) = i$ and $s(\beta,t+1) = 1-s(\beta,t)$;
\item with probability $r-m$ (a $Z_{i,t} =0$ case), $p(\alpha,t+1) = i$, $s(\alpha,t+1) = 1-s(\alpha,t)$, $p(\beta,t+1) = i+1$ and $s(\beta,t+1) = 1-s(\beta,t)$.
\end{enumerate}
Moreover, all these transitions are independent. Then, the law of $((p(\prt{i},t),s(\prt{i},t)) : i \in \ZZ, t \in \NN)$ is $\mathcal{P}(\mu;p,r)$.
\end{lemma}

\begin{proof}
To prove the lemma, we have to show that where we are in the case 1, we can obtain with probability $1/2$ and $1/2$ conclusion of case 2 and 4. That is the case because $B_{i,t} = \begin{cases} s(\alpha,t) &  \text{w.p. }1/2,\\ 1-s(\alpha,t) & \text{w.p. } 1/2. \end{cases}$
\end{proof}

In the following,  for any particle $\alpha$, any $t,t_1,t_2\in \NN$ with $t_1< t_2$, we denote the event $I(\alpha,t) = \{\omega : Z_{p(\alpha_i,t),t} = 1\}$: ``particle $\alpha$ do the transition 1 at time $t$ (from time $t$ to $t+1$)'' and we denote $I_2(\alpha,t_1,t_2)$ the event: ``$\alpha$ do the transition 1 at time $t_1$ and $t_2$ with two different neighbor particles (i.e. $I(\alpha,t_1)$ and $I(\alpha,t_2)$ and $\alpha(p(\alpha,t_1) + (-1)^{p(\alpha,t_1)+t_1},t_1) \neq \alpha(p(\alpha,t_2) + (-1)^{p(\alpha,t_2)+t_2},t_2)$)''. Finally, we denote 
\begin{displaymath}
I_2(\alpha,t) = \bigcup_{(t_1,t_2) : 0 \leq t_1 < t_2 < t} I_2(\alpha,t_1,t_2), 
\end{displaymath}
the event that $\alpha$ do at least once the transition 1 with two different neighbor particles before time $t$.
Now, the correlation function of the 8-vertex model with initial law $\mu$ can be rewritten in this system of particles. First, we need the following lemma.
\begin{lemma} \label{lem:indVitesse}
Let $((p(\prt{i},t),s(\prt{i},t)) : i \in \ZZ, t \in \NN) \sim \mathcal{P}(\mu;p,r)$. Then, for any $i \in \ZZ$, for any $t$, $I_2(\alpha_i,t)$ is independent of $s(\alpha_0,0)$.
\end{lemma}

\begin{proof}
This is induce by the fact that for any $i \in \ZZ$, $I(\alpha_i,t)$ is independent of $s(\alpha_0,0)$, that is true because $Z$ is independent of $(s(\alpha_i,0): i \in \ZZ)$.
\end{proof}

Now, we can express the correlation function of the 8-vertex model with any boundary condition and $a+c=b+d$.
\begin{lemma} \label{lem:CitVitesse}
Let $\mu$ be any probability measure on $\{0,1\}^\ZZ$ and let $O \sim \mathcal{L}(\mu;p,r)$. Then the correlation function $C$ satisfies
\begin{equation}
\left| \sqrt{\frac{\var{e(i,t)}}{\var{e(0,0)}}}\, C((0,0);(i,t)) - C_8(i,t) \right|  \leq 2 \prob{I_2(\alpha(i,t),t)^c}
\end{equation}
\end{lemma}

\begin{proof}
Let $O = (e(i,t) : i \in \ZZ, t \in \NN) \sim \mathcal{L}(\mu;p,r)$. Then,
\begin{align*}
\cov{e(0,0)}{e(i,t)} & = \prob{e(0,0) = 1 = e(i,t) } - \prob{e(0,0) = 1} \prob{e(i,t)=1}\\
& = \prob{e(0,0)=1} \big(\prob{e(i,t) = 1 | e(0,0) = 1} - \prob{e(i,t)=1|e(0,0)=1} \prob{e(0,0)=1} \nonumber\\
& \qquad - \prob{e(i,t)=1|e(0,0)=0} \prob{e(0,0)=0} \big) \\
& = \var{e(0,0)} (\prob{e(i,t) = 1 | e(0,0) = 1} - \prob{e(i,t) = 1 | e(0,0) = 0} )\\
& = \var{e(0,0)} (\prob{s(\alpha(i,t),t) = 1 | s(\alpha_0,0) = 1} - \prob{s(\alpha(i,t),t) = 1 | s(\alpha_0,0) = 0} ).
\end{align*}
Hence,
\begin{align*}
& \sqrt{\frac{\var{e(i,t)}}{\var{e(0,0)}}}\, C((0,0);(i,t)) \\
& \qquad  = \underbrace{\prob{s(\prt{0},t) = 1 \text{ and } \alpha(i,t) = \prt{0}| s(\prt{0},0) = 1} - \prob{s(\prt{0},t) = 1 \text{ and } \alpha(i,t) = \prt{0} | s(\prt{0},0) = 0}}_{=C_8(i,t)} \nonumber\\
& \qquad \quad + \prob{s(\alpha(i,t),t) = 1 \text{ and } \alpha(i,t) \neq  \prt{0} \text{ and } I_2(\alpha(i,t),t) | s(\prt{0},0) = 1} \nonumber \\
& \qquad \quad \quad \quad - \prob{s(\alpha(i,t),t) = 1 \text{ and } \alpha(i,t) \neq \prt{0} \text{ and } I_2(\alpha(i,t),t) | s(\prt{0},0) = 0}\nonumber \\
& \qquad \quad + \prob{s(\alpha(i,t),t) = 1 \text{ and } \alpha(i,t) \neq  \prt{0} \text{ and } I_2(\alpha(i,t),t)^{c} | s(\prt{0},0) = 1} \nonumber \\
& \qquad \quad - \prob{s(\alpha(i,t),t) = 1 \text{ and } \alpha(i,t) \neq  \prt{0} \text{ and } I_2(\alpha(i,t),t)^{c} | s(\prt{0},0) = 0}.
\end{align*} 
But,
\begin{align*}
& \prob{s(\alpha(i,t),t) = 1 \text{ and } \alpha(i,t) \neq  \prt{0} \text{ and } I_2(\alpha(i,t),t) | s(\prt{0},0) = 1} \\
& \quad = \frac{1}{2} \prob{\alpha(i,t) \neq  \prt{0} \text{ and } I_2(\alpha(i,t),t) | s(\prt{0},0) = 1} \text{(because $I_2(\alpha(i,t),t) \Rightarrow s(\alpha(i,t),t) \sim \bern{1/2}$)}\\
& \quad = \frac{1}{2} \prob{\alpha(i,t) \neq  \prt{0} \text{ and } I_2(\alpha(i,t),t)} \\
& \qquad \text{(by Lemma~\ref{lem:indVitesse} and the fact that trajectories of particles are independent of their initial states)}.
\end{align*}
Similarly,
\begin{displaymath}
\prob{s(\alpha(i,t),t) = 1 \text{ and } \alpha(i,t) \neq  \prt{0} \text{ and } I_2(\alpha(i,t),t) | s(\prt{0},0) = 0}  = \frac{1}{2} \prob{\alpha(i,t) \neq  \prt{0} \text{ and } I_2(\alpha(i,t),t)}.
\end{displaymath}

Hence,
\begin{align*}
\left| \sqrt{\frac{\var{e(i,t)}}{\var{e(0,0)}}}\, C((0,0);(i,t)) - C_8(i,t)\right| & \leq \prob{I_2(\alpha(i,t),t)^{c} | s(\prt{0},0)=1} + \prob{I_2(\alpha(i,t),t)^{c} | s(\prt{0},0)=0} \\
& = 2 \prob{I_2(\alpha(i,t),t)^{c}} \text{(by Lemma~\ref{lem:indVitesse})}\\
& = 2 \prob{I_2(\alpha_0,t)^{c}}
\end{align*}
because event $I_2(\prt{i},t)$ occurs with the same probability for any particle $\prt{i}$.
\end{proof}

Now, to prove Proposition~\ref{prop:vitesse}, we have to find a bound on $\prob{I_2(\prt{0},t)^c}$.

\begin{lemma} \label{lem:vitesse}
Let $m = \min(p,r)$.
\begin{equation}
\prob{I_2(\prt{0},t)^c} \leq (1-m)^{t-1-\lfloor t/2 \rfloor} + (1-m)^{\lfloor t/2 \rfloor} - (1-m)^{t-1} 
\end{equation}
\end{lemma}

\begin{proof}
We denote $T_1(\omega) = \inf \{t : \omega \in I(\prt{0},t)\}$. By definition of $I(\prt{0},t)$, $T_1 \sim \text{Geo}(2m)$ (i.e. it is a geometric random variable whose probability of success is $2m$). Now, denote $T_2(\omega) = \inf \{t : \omega \in I_2(\prt{0},T_1(\omega),t)\}$. Due to the fact that two particles cannot interact twice during two successive steps of time, $\left \lceil \frac{T_2-T_1}{2} \right \rceil \leq T$ in law where $T \sim \text{Geo}(2m)$. Hence,
\begin{align*}
\prob{I_2(\prt{0},t)} & \geq \prob{T_2 \leq t}\\
& \geq \prob{2T + T_1 \leq t} \\
& = \sum_{t'=1}^{\lfloor t/2 \rfloor} \prob{T=t'} \prob{T_1 \leq t-2t'}\\ 
& = \sum_{t'=1}^{\lfloor t/2 \rfloor} 2m (1-2m)^{t'-1} (1-(1-2m)^{t-2t'}) \\
& = 2m \sum_{t'=1}^{\lfloor t/2 \rfloor} \left( (1-2m)^{t'-1} - (1-2m)^{t-t'-1} \right) \\
& = 2m \left(\frac{1 - (1-2m)^{\lfloor t/2 \rfloor}}{2m} - (1-2m)^{t-1} \frac{(1-2m)^{-{\lfloor t/2 \rfloor}}-1}{2m} \right)\\
& = 1 - (1-2m)^{\lfloor t/2 \rfloor} + (1-2m)^{t-1} - (1-2m)^{t - \lfloor t/2 \rfloor -1}.
\end{align*}
\end{proof}

Now, we can do the proof of Proposition~\ref{prop:vitesse}.

\begin{proof}[Proof of Proposition~\ref{prop:vitesse}]
In the case $p+r \leq 1$, by lemmas~\ref{lem:CitVitesse} and~\ref{lem:vitesse},
\begin{equation} \label{eq:vitesseP}
\left| \sqrt{\frac{\var{e(i,t)}}{\var{e(0,0)}}}\, C((0,0),(i,t)) - C_8(i,t)\right| \leq  2 \left((1-2m)^{\lfloor t/2 \rfloor} + (1-2m)^{t - \lfloor t/2 \rfloor -1} - (1-2m)^{t-1}\right)
\end{equation} 
where $m = \min(p,r)$.

In the case $p+r \geq 1$, a similar proof permits to find~\eqref{eq:vitesseP} with $m = \min(1-p,1-r)$.

To conclude, we remark that when $p+r \leq 1$, then $p \leq 1-r$ and $r \leq 1-p$, so $\min(1-p,1-r,p,r) = \min(p,r)$; and when $p+r \geq 1$, $\min(p,r,1-p,1-r) = \min(1-p,1-r)$. And, finally, we remark that $1-2 \min(p,r,1-p,1-r) = \lambda(p,r)$. That is ending the proof.
\end{proof}

\section{Asymptotic of $C_8(i,t)$: proof of Theorem~\ref{thm:asympt}} \label{sec:asympt}
In this Section, we suppose that $p+r \neq 1$. Case $p+r=1$ has been treated in Remark~\ref{rem:EASY}. Proof of Theorem~\ref{thm:asympt} is done in two steps. In a first step, the asymptotic is proved when $b=0$ (i.e. when $r=0$ and $0 \leq p < 1$). And, in a second step, it is generalized for any $(a,b,c,d)$ such that $a+c=b+d$.

\subsection{Case $r=0$ and $0 \leq p <1$} \label{sec:asympt0}
In this case, $\lambda(p,0) = 1$, see~\eqref{eq:lambda}, and, so, Theorem~\ref{thm:asympt} is
\begin{proposition} \label{prop:aux}
If $r=0$ and $0 \leq p <1$, then there exists $c>0$ such that, for any $t\in \NN$, for any $i \in \ZZ$
\begin{displaymath}
C_8(i,t) \leq \frac{c}{\sqrt{t}}.
\end{displaymath}
\end{proposition}

To prove this proposition, we prove first two lemmas on asymptotic behavior of random walks.
\begin{lemma} \label{lem:MAsimple}
Let $S = (S_t : t \in \NN)$ be a simple random walk on $\ZZ$, i.e.
\begin{itemize}[topsep=0pt,itemsep=0pt,partopsep=0pt,parsep=0pt]
\item $S_0 = 0$ a.s.,
\item for any $t \geq 0$, 
\begin{displaymath}
\prob{S_{t+1} - S_t = 1} = \prob{S_{t+1} - S_t = -1} = 1/2.
\end{displaymath} 
\end{itemize}
Then, there exists a constant $c>0$ such that for any $t \in \NN$, for any $i \in \ZZ$,
\begin{displaymath}
\prob{S_t=i} \leq \frac{c}{\sqrt{t}}.
\end{displaymath}
\end{lemma}

\begin{proof}
It is a classical result in probability theory. Proofs of generalization of this lemma exists for sum of i.i.d.\ random variables, see~\cite[Chapter~3]{Petrov75}. A proof in that simple case can be obtained by enumeration of binary paths and application of Stirling's formula.
\end{proof}

Second lemma generalizes Lemma~\ref{lem:MAsimple} to some processes constructed with a simple random walk.

\begin{lemma} \label{lem:MAnotsimple}
Let $X = (X_t : t \in \NN)$ a process with values in $\ZZ$. If, for any $t \in \NN$,
\begin{displaymath}
X_t \eqd S_{N_t} + R_{t,{N_t}}
\end{displaymath}
such that $N = (N_t : t \in \NN)$ where $N_t$ follows a binomial laws of parameters $(t,q)$ ($q \neq 0$), $S = (S_u : u \in \NN)$ is a simple random walk on $\ZZ$ and $R = \left( R_{t,u} : t \in \NN, u \in \NN \right)$ is any collection of any random variables on $\ZZ$ and $(N,S,R)$ are independent, then, there exists $c>0$ such that for any $i$, for any $t$,
\begin{displaymath}
\prob{X_t = i} \leq \frac{c}{\sqrt{t}}.
\end{displaymath}
\end{lemma}

\begin{proof}
Let $t \in \prive{\NN}{\{0\}}$, $i \in \ZZ$ and $\epsilon >0$,
\begin{align*}
\prob{X_t = i} & = \prob{S_{N_t} + R_{t,{N_t}} = i} \\
& \leq \prob{|N_t-qt| > qt/2} + \max_{n \in [qt/2,3qt/2]} \prob{S_n + R_{t,n} = i} \\
& \leq \frac{q(1-q)t}{(q t/2)^2} + \max_{n \in [qt/2,3qt/2]} \prob{S_n = i - R_{t,n}} \text{\ (by Chebyshev's inequality)} \\
& \leq \frac{4(1-q)}{q t} + \max_{n \in [qt/2,3qt/2]} \frac{c}{\sqrt{n}} \text{\ (by Lemma~\ref{lem:MAsimple})}\\
& \leq \frac{c_1}{\sqrt{t}} + \frac{c \sqrt{2}}{\sqrt{q}} \frac{1}{\sqrt{t}} \\
& = \frac{c_2}{\sqrt{t}}.
\end{align*}
\end{proof}

Now, we define a homogeneous Markov chain $Y$ with values on $\ZZ$.
\begin{definition}
The process $Y = (Y_t:t \in \NN)$ have the following properties:
\begin{itemize}[topsep=0pt,itemsep=0pt,partopsep=0pt,parsep=0pt]
\item $Y_0 = 0$ a.s.;
\item for any $t \in \NN$,
\begin{displaymath}
Y_{t+1} = Y_t + 
\begin{cases}
-1 & \text{w.p. } p(1-p),\\
0 & \text{w.p. } p^2, \\
1 & \text{w.p. } p(1-p), \\
2 (-1)^{Y_t} & \text{w.p. } (1-p)^2.
\end{cases}
\end{displaymath}
\end{itemize}
\end{definition}

This Markov chain is related to the Markov chain $X$ defined in Definition~\ref{def:XMA}.

\begin{proposition} \label{prop:X2t}
Let $p \in [0,1)$ and $r=0$. In that case, the non-homogeneous Markov chain $X$ (defined in Definition~\ref{def:XMA}) satisfies, for any $t \in \NN$, 
\begin{equation}
X_{2t} \eqd (Y_t,1) \text{ and } X_{2t+1} \eqd \begin{cases} (1+Y_t,0) & \text{w.p. } 1-p, \\ (-Y_t ,0 ) & \text{w.p. } p. \end{cases}
\end{equation}

In particular, \eqref{eq:COT2} gives, for any $t$,
\begin{align}
C_8(i,2t) & = \prob{Y_t = 0} \text{ and }\\
C_8(i,2t+1) & = - \left(p\ \prob{Y_t = 0} + (1-p)\ \prob{Y_t = -1}\right).
\end{align}
\end{proposition}

\begin{proof}
We suppose that $r=0$, we obtain, for any $t \in \ZZ$, applying twice \eqref{eq:XMA}, that: if $X_t = (i,k)$, then
\begin{displaymath}
X_{t+2} = 
\begin{cases}
(i-1,k) & \text{w.p. } p(1-p),\\
(i,k) & \text{w.p. } p^2,\\
(i+1,k) & \text{w.p. } p(1-p),\\
(i+2 (-1)^{i+t},k) & \text{w.p. } (1-p)^2.
\end{cases}
\end{displaymath}

As $X_0 = (0,1)$, we obtain that, for any $t \in \ZZ$, first coordinate of $X_{2t}$ is equal in distribution to $Y_t$ and its second coordinate is $1$ a.s.
And, as 
\begin{displaymath}
X_1 = \begin{cases} (0,0) & \text{w.p. } p, \\ (1,0) & \text{w.p. } 1-p, \end{cases}
\end{displaymath}
first coordinate of $X_{2t+1}$ is equal in distribution to $-Y_t$ w.p.\ $p$ and to $1+Y_t$ w.p.\ $1-p$ and its second coordinate is $0$ a.s.
\end{proof}

Now, we decompose $Y$ so that $Y$ satisfies conditions of Lemma~\ref{lem:MAnotsimple}. We define first a law on $\ZZ$.
\begin{definition} \label{def:RTN}
Let $p \in [0,1]$, $q \in [0,1]$, $t \in \NN$, $n \in \NN$. Let $(L_j : 0 \leq j \leq n)$ be $n+1$ i.i.d.\ random variables distributed according to geometric law of parameter $q$, i.e., for any $j \in \NN$, for any $k \in \NN$,
\begin{displaymath}
\prob{L_j = k} = (1-q)^k q.
\end{displaymath}
We denote by $\mathcal{L}_{t,n}(q)$ the law on $\NN^{n+1}$ of $(L_j : 0 \leq j \leq n)$ conditioned by $\sum_{j=0}^n L_j = t-n$. Let $(L_j : 0 \leq j \leq n) \sim \mathcal{L}_{t,n}(q)$. For any $j$, we set $G_j$ to be a random variable distributed according to binomial law of parameters $(L_j,p)$, and we suppose that $(G_j : 0 \leq j \leq n)$ knowing $(L_j: 0 \leq j \leq n)$ are independent. We define then by $\mathcal{R}_{t,n}(p,q)$ the law of
\begin{displaymath}
R_{t,n} = \sum_{j=0}^{n} (-1)^j G_j.
\end{displaymath}
\end{definition}

Now, we can check that $Y$ satisfies conditions of Lemma~\ref{lem:MAnotsimple}.
\begin{lemma} \label{lem:eqloi}
Process $Y$ satisfies, for any $t \in \NN$, $Y_t \eqd S_{N_t} + 2 R_{t,{N_t}}$ where
\begin{itemize}[topsep=0pt,itemsep=0pt,partopsep=0pt,parsep=0pt]
\item $N_t$ follows a binomial law of parameters $(t,2p(1-p))$,
\item $S = (S_u : u \in \NN)$ is a simple random walk on $\ZZ$,
\item $R = (R_{t,n})$ is a collection of (independent) random variables of mono-dimensional law: for any $t$, for any $n$, $R_{t,n}$ is distributed according to $\mathcal{R}_{t,n} \left( \frac{(1-p)^2}{p^2+(1-p)^2},2p(1-p) \right)$,
\item $N$, $S$ and $R$ are mutually independent.
\end{itemize} 
\end{lemma}

\begin{proof}
Let $t \in \NN$ and $Y_t$. For any $1 \leq i \leq t$, we denote $\Delta(Y_u) = Y_u-Y_{u-1}$. We remark that $\left(\Delta(Y_u) : 1 \leq u \leq t \right)$ is a sequence of independent random variables of law
\begin{displaymath}
\Delta(Y_i) = 
\begin{cases}
-1 & \text{w.p. } p(1-p), \\
0 & \text{w.p. } p^2, \\
1 & \text{w.p. } p(1-p), \\
-2 & \text{w.p. } (1-p)^2 \text{ if } Y_i \text{ if odd}, \\
2 & \text{w.p. } (1-p)^2 \text{ if } Y_i \text{ is even}.
\end{cases}
\end{displaymath}

As $Y_0 = 0$, we have $Y_t = \sum_{i=1}^t \Delta(Y_i)$.

First, we study the set  
\begin{displaymath}
E_1 = \{u : 1 \leq u \leq t \text{ and } | \Delta(Y_u) | = 1 \}.
\end{displaymath}
We denote by $(u_1,\dots,u_N)$ the elements of $E_1$ sorted in increasing order. $E_1$ is the set of instants for which $Y_{t_i-1}$ and $Y_{t_i}$ have different parities. Cardinal $N$ of $E_1$ follows a binomial law of parameters $(t,2p(1-p))$ (indeed, at each step of times, $\Delta(Y_i) = \pm 1$ with probability $2p(1-p)$). The random variable $S_N = \sum_{i=1}^N \Delta(Y_{t_i})$ is distributed as a simple random walk finishing at a random time $N$. 

We insist on the fact that, for any $0 \leq j \leq N-1$, every element of $(Y_i : u_{j} \leq i \leq u_{j+1}-1)$ (setting $u_0=0$) are of the same parity as $j$, in particular $|Y_{u_{j+1}-1} - Y_{u_j}|$ is even.
Let, for any $j \geq 0$, $L_j = u_{j+1}-1 - u_{j}$. By construction of $Y$, random variables $(L_j : 0 \leq j \leq N)$ follow the law $\mathcal{L}_{t,N}(2p(1-p))$ (see Definition~\ref{def:RTN}). 
We denote $G_j = |Y_{u_{j+1}-1} - Y_{u_j}|/2$. By construction of $Y$,
\begin{displaymath}
G_j = \sum_{i=1}^{L_j} X^{(j)}_i 
\end{displaymath}
where $(X^{(j)}_i : 1 \leq i \leq L_j , 0 \leq j \leq N)$ are i.i.d. of law: for any $i,j$,
\begin{displaymath}
\prob{X^{(j)}_i = 1} = 1 - \prob{X^{(j)}_i = 0} = \frac{(1-p)^2}{p^2+(1-p)^2}.
\end{displaymath}
In other words, $G_j$ follows a binomial law of parameters $\left(L_j, \frac{(1-p)^2}{p^2+(1-p)^2} \right)$.

Hence, we effectively obtain that
\begin{displaymath}
Y_t \eqd S_N +  2 \underbrace{\sum_{j=0}^N (-1)^j G_j}_{R_{t,N}}
\end{displaymath} 
where $N$ follows a binomial law of parameter $(t,2p(1-p))$.
\end{proof}

\begin{proof}[Proof of Proposition~\ref{prop:aux}]
Lemmas~\ref{lem:MAnotsimple},~\ref{lem:eqloi} and Proposition~\ref{prop:X2t} have for immediate consequence Proposition~\ref{prop:aux}.
\end{proof}

\subsection{General case: proof of Theorem~\ref{thm:asympt}} \label{sec:asymptg}
In this section, we prove only case $i=0$. Cases $i \neq 0$ could be proved in a similar way but with some sections more technical. We denote by $H$ the function defined by, for any $(p,r) \in [0,1]^2$ such that $p+r \neq 1$,
\begin{equation}
H(p,r) = \frac{(1-2p)(1-2r)}{(1-(p+r))^2}
\end{equation}
and we let, for any $n \geq 0$,
\begin{align}
M^{(0)}_{n}(X) & = \sum_{k=0}^{n} (-1)^{k} \binom{2n-1-k}{n-k} \binom{n}{k} X^k \text{ and }\\
M^{(1)}_{n}(X) & = \sum_{k=0}^{n} (-1)^{k} \binom{2n-k}{k,n-k,n-k} X^k.
\end{align}

With these notation, when $p+r \neq 1$, $C_8(0,t)$ (in Theorem~\ref{eq:Cit}) becomes, for any $t$,
\begin{itemize}[topsep=0pt,itemsep=0pt,partopsep=0pt,parsep=0pt]
\item if $t$ is even,
\begin{equation} \label{eq:C0ti}
C_8(0,t) = (1-(p+r))^t \ M^{(0)}_{t/2}\left( H(p,r) \right),
\end{equation}

\item if $t$ is odd,
\begin{equation} \label{eq:C0tp}
C_8(0,t) = (r-p) (1-(p+r))^{t-1} \ M^{(1)}_{(t-1)/2} \left( H(p,r) \right).
\end{equation}
\end{itemize}

Hence, the asymptotic of $C_8(0,t)$ as $t \to \infty$ is related to those, as $n \to \infty$, of sequences of polynomials $\left(M^{(0)}_n(X) : n \geq 0\right)$ and $\left(M^{(1)}_n(X): n \geq 0\right)$ when $X=H(p,r)$. To evaluate those asymptotic behaviors, we let the function $m$ that is, for any $K \leq 1$,
\begin{equation}
m(K) = \frac{1}{1+\sqrt{1-K}}.
\end{equation}

\begin{lemma} \label{lem:asymptPoly}
For any $K \leq 1$, when $n \to \infty$,
\begin{align}
M^{(0)}_{n}(K) & = O\left(\frac{m(K)^{-2n}}{\sqrt{2n}}\right) \text{ and }\\
M^{(1)}_n(K) & = O\left(\frac{m(K)^{-2n}}{\sqrt{2n}}\right). \\
\end{align}
\end{lemma}

\begin{proof}
By Proposition~\ref{prop:aux} applied to $p=1-m(K)$ and $r=0$ and~\eqref{eq:C0ti}, for any $t=2n$ even,
\begin{displaymath}
m(K)^{2n} \ M^{(0)}_{n}\!\left(H(1-m(K),0)\right) = C_8(0,2n) = O\left(1/\sqrt{2n}\right).
\end{displaymath}
As $H(1-m(K),0) = K$, multiplying by $m(K)^{-2n}$,
\begin{displaymath}
M^{(0)}_{n}\!\left( K \right) = O\left(\frac{m(K)^{-2n}}{\sqrt{2n}}\right).
\end{displaymath}

We prove asymptotic of $M^{(1)}_n(K)$ as $n \to \infty$ with same arguments.
\end{proof}

\begin{proof}[Proof of Theorem~\ref{thm:asympt}]
Equations~\eqref{eq:C0ti} and~\eqref{eq:C0tp} and Lemma~\ref{lem:asymptPoly} imply that, for any $t$, 
\begin{equation}
C_8(0,t) = O\left( \frac{(1-(p+r))^t \  m(H(p,r))^{-t}}{\sqrt{t}} \right) = O\left( \frac{\lambda(p,r)^t}{\sqrt{t}} \right)
\end{equation}
where $\displaystyle \lambda(p,r) = \left| \frac{1-(p+r)}{m(H(p,r))}\right|$. Now, let us compute $\lambda(p,r)$.

First, we compute $m(H(p,r))$.
\begin{align*}
m(H(p,r)) & = \frac{1}{1+\sqrt{1- \frac{(1-2p)(1-2r)}{(1-(p+r))^2}}} \\
& = \frac{|1-(p+r)|}{|1-(p+r)|+|p-r|}
\end{align*}
This last quantity is $\displaystyle \frac{1-(p+r)}{1-2r}$ if ($1-(p+r) \geq 0$ and $p-r \geq 0$) or $(1-(p+r) \leq 0$ and $p-r \leq 0$) (i.e. if $(1-(p+r))(p-r) = p-p^2 - (r-r^2)  \geq 0$), and it is $\displaystyle \frac{1-(p+r)}{1-2p}$ else. Hence,
\begin{displaymath}
m(H(p,r)) =
\begin{cases}
\displaystyle \frac{1-(p+r)}{1-2p} & \text{if } p(1-p) \leq r(1-r), \\
\displaystyle \frac{1-(p+r)}{1-2r} & \text{if } r(1-r) \leq p(1-p).
\end{cases}
\end{displaymath}

And so,
\begin{displaymath}
\lambda(p,r) = 
\begin{cases}
|1-2p| & \text{if } p(1-p) \leq r(1-r), \\
|1-2r| & \text{if } r(1-r) \leq p(1-p).
\end{cases}
\end{displaymath}

As $p+(1-p)=1=r+(1-r)$, we can use classical results about areas of rectangles with same perimeter, to conclude that $\lambda(p,r) =  \max(|1-2p|,|1-2r|)$.
\end{proof}

\begin{remark}
The value of $m(K)$ has been originally obtained by studying the conic equation $\mathcal{H}_K = \{(p,r) : H(p,r) = K\}$ that is the union of two lines.
\end{remark}

\section{Vertex models and triangular probabilistic cellular automata} \label{sec:PCA}
In this section, we show that 8-vertex model on $\overline{K}_\infty$ when $a+c=b+d$ can be obtain as the space-time diagram of a new type of probabilistic cellular automata, called here triangular PCA. Using the theory of probabilistic cellular automata and the fact that some of them are ergodic~\cite{CM17}, we can prove Proposition~\ref{prop:ergo}.

\subsection{Triangular probabilistic cellular automata} \label{sec:PCADT} 
We define first probabilistic cellular automata (PCA) of order $2$ whose triangular PCA (TPCA) are special cases.
A PCA $\PCA{A}$ of order $2$ is a quintuple $(E,\LL,N_1,N_2,T)$ where:
\begin{itemize}[topsep=0pt,itemsep=0pt,partopsep=0pt,parsep=0pt]
\item $E$ is a finite set;
\item $\LL$ is a lattice;
\item $N_1$ is a neighborhood function of $\LL$, i.e.\ there exists a finite subset $I_1$ of $\LL$ such that, for any $i \in \LL$, $N_1(i) = (i+j: j \in I_1)$, we denote $|N_1|$ the cardinal of $I_1$;
\item $N_2$ is another neighborhood function of $\LL$ and 
\item $T$ is a transition matrix (t.m.) from $E^{|N_2|} \times E^{|N_1|}$ to $E$, i.e.\ for any $(x,y) \in E^{|N_2|} \times E^{|N_1|}$, for any $z \in E$, $T(x,y;z) \geq 0$, and $\sum_{z' \in E} T(x,y;z') = 1$.
\end{itemize}
From this quintuple, we define a Markov chain $(S_t:t \geq 0)$ of order $2$ on $E^\LL$ in the following way: for any subset $C \subset \LL$, for any $(z_i : i \in C) \in E^C$,
\begin{align}
& \prob{\ \left(S_{t+2}(i) = z_i : i \in C \right)~|~S_t=(x_i:i\in \LL) , S_{t+1}=(y_i:i\in \LL)\ } \nonumber \\
& \qquad = \prod_{i \in C} T( (x_j:j \in N_2(i)),(y_{j'}: j'\in N_1(i)) ; z_i).
\end{align}
The process $S_{t+2}$ is well defined because its law is defined on a compatible way on all cylinders of $E^\LL$. $(S_t(i) : i \in \LL, t \geq 0)$ is called \emph{space-time diagram} of $\PCA{A}$.\par

An other way to see PCA of order $2$ is to consider them as a deterministic map from $\mes{E^{\LL} \times E^{\LL}}$ (the set of probability measure on $E^{\LL} \times E^{\LL}$) to $\mes{E^{\LL} \times E^{\LL}}$. Let $\PCA{A}$ be a PCA $(E,\LL,N_1,N_2,T)$. Let $\mu \in \mes{E^{\LL} \times E^{\LL}}$ and $(S_{t_0},S_{t_0+1}) \sim \mu$. We denote by $\nu$ the law of $(S_{t_0+1},S_{t_0+2})$ where $S_{t_0+2}$ is the image of $(S_{t_0},S_{t_0+1})$ by $\PCA{A}$. Then, for any subset $C \in \LL$ and any $(y,z) \in E^\LL \times E^\LL$,
\begin{align*}
& \nu((y_i: i \in N_1(C)),(z_i:i \in C)) \nonumber \\
& = \sum_{(x_i:i \in N_2(C)) \in E^{N_2(C)}} \mu((x_i:i\in N_2(C)),(y_i \in N_1(C))) \prod_{i \in C} T((x_j:j \in N_2(i)),(y_j:j \in N_1(i));z_i)
\end{align*}
where $N_k(C) = \cup_{i \in C} N_k(i)$ for any $k \in \{1,2\}$. We denote by $\tr{A}$ the function that maps $\mu$ to $\nu = \tr{A}(\mu)$. We say that $\mu$ is an \emph{invariant probability measure} (i.p.m.) of $\PCA{A}$ if $\mu=\tr{A}(\mu)$.\par

In the following, we consider only the cases where $\LL = \ZZ$, $N_1(i) = (i,i+1)$ and $N_2(i) = (i+1)$. Such PCA of order $2$ are called, in this article, triangular probabilistic cellular automata (TPCA). The name comes from the fact that their space-time diagrams are triangular lattices (see Figure~\ref{fig:ACPVet}). To simplify reading, transitions $T((x_{i+1}),(y_i,y_{i+1});z_i)$ of TPCA are denoted now $T(y_i,x_i,y_{i+1};z_i)$.\par

\begin{figure}
\begin{center}
\begin{tabular}{cc}
\includegraphics{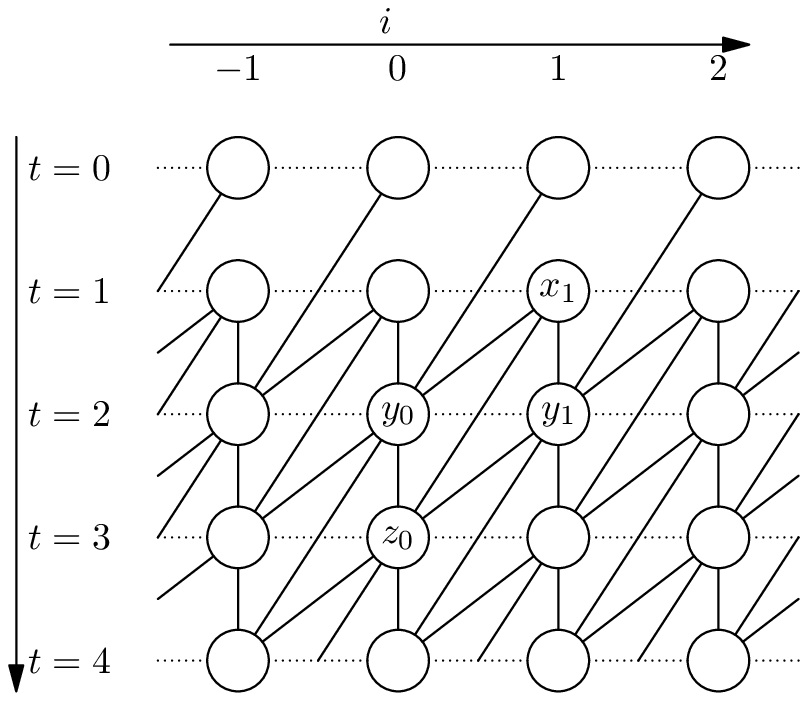} & \includegraphics{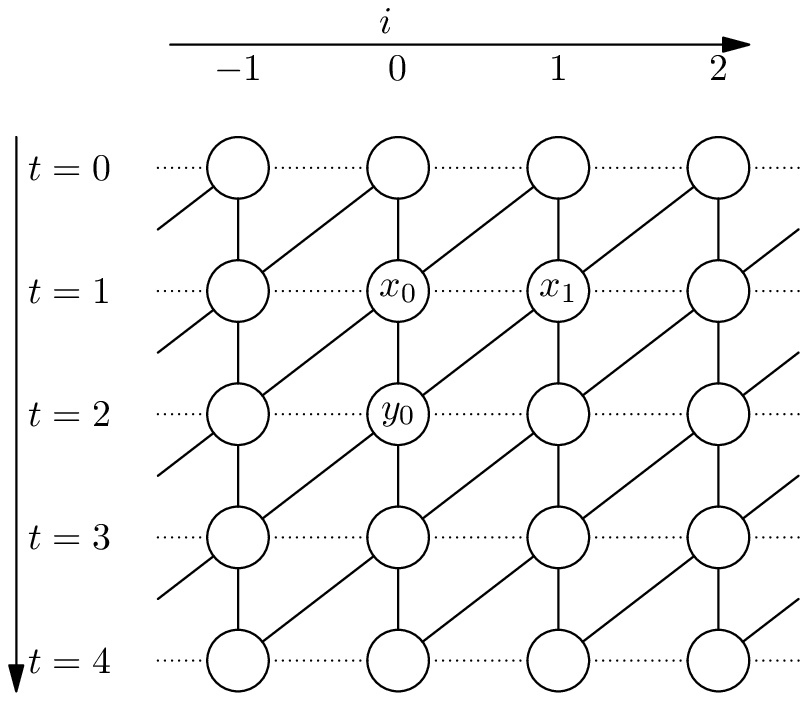}
\end{tabular}
\end{center}
\caption{Left: (empty) space-time diagram of a TPCA. Right: (empty) space-time diagram of a SPCA.}
\label{fig:ACPVet}
\end{figure}

\medskip

Before seeing new results on TPCA, we recall some theorems about ``classical'' PCA. The ``classical'' PCA, considered here, are PCA of order $2$ for which $N_1(i)= (i,i+1)$ and $N_2(i) = \emptyset$. We will call them \emph{square PCA} (SPCA) in the following because their space-time diagram are homeomorphic to $\ZZ^2$ (see Figure~\ref{fig:ACPVet}).\par

Cellular automata and ``classical'' PCA have been studied since 1940s. For more information on PCA, we refer the interested reader to the recent survey of Mairesse and Marcovici~\cite{MM14}. In the present work, we focus our attention on results on SPCA whose one of its invariant probability measures is a Markovian distribution~\cite{BGM69,Vasilyev78,TVSMKP90,DPLR02,B-M98,MM14IHP,CM15,Casse16}. In particular, we need to recall Theorem~2.6 of~\cite{CM15} that characterizes SPCA whose one of its invariant probability measures is a $(D,U)$-HZMC (Horizontal Zigzag Markov Chain).\par

A law $\mu$ on $E^\ZZ \times E^\ZZ$ is a $(D,U)$-HZMC distribution if there exists a pair $(D,U)$ of stochastic matrices from $E$ to $E$ and a family $(\rho_i: i \in \ZZ)$ of probability measures on $E$ such that, for any $k_1, k_2 \in \ZZ$, $k_1 < k_2$, for any $(x_i : k_1 \leq i \leq k_2)$, $(y_i : k_1 \leq i \leq k_2 - 1)$,
\begin{equation}
\mu((x_i: k_1 \leq i \leq k_2),(y_i:k_1\leq i \leq k_2 -1)) = \rho_{k_1}(x_{k_1}) \prod_{i=k_1}^{k_2-1} D(x_j;y_j) U(y_j;x_{j+1}) 
\end{equation}
and, for any $i \in \ZZ$ and $x_{i+1} \in E$,
\begin{equation}
\rho_{i+1}(x_{i+1}) = \sum_{x_i \in E} \rho_i(x_i) \sum_{y_i \in E} D(x_i;y_i) U(y_i;x_{i+1}). 
\end{equation}
In other words, a HZMC distribution is a Markovian distribution on states of two consecutive lines crossed from bottom to top and left to right (see Figure~\ref{fig:HZMC}). In the following, we denote $(x,y) \sim \HZ{(D,U)}$ if $(x,y)$ is distributed according to a $(D,U)$-HZMC distribution.\par

\begin{figure}
\begin{center}
\begin{tabular}{cc}
\includegraphics{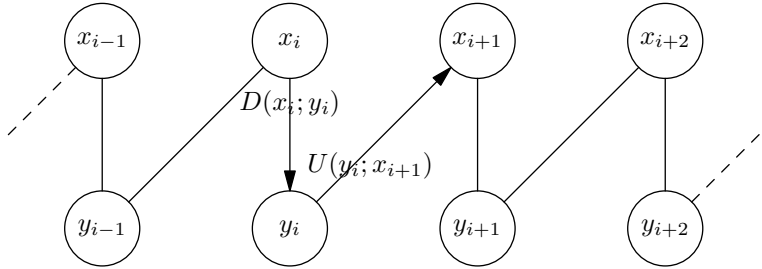}
\end{tabular}
\end{center}
\caption{Representation of a horizontal zigzag Markov chain (HZMC)}
\label{fig:HZMC}
\end{figure}

Now, we define some quantities needed to state Theorem~2.6 of~\cite{CM15}. Let $T$ be any stochastic Markov kernel from $E^2$ to $E$ with positive coefficients. 
Let $\nu = (\nu(x) : x \in E)$ be the stochastic (i.e.\ normalized such that $\sum_{x \in E} \nu(x) = 1$) left eigenvector associated to the eigenvalue $1$ of the following stochastic matrix 
\begin{displaymath}
\left(T(x,x;y) : x \in E,y \in E \right)
\end{displaymath}
(this eigenvector is unique due to the Perron-Frobenius Theorem) and $\gamma$ be the stochastic left eigenvector of the matrix
\begin{displaymath}
\left(\nu(y) \frac{T(y,y;0)}{T(y,x;0)} : x \in E, y\in E \right)
\end{displaymath}
associated with $\lambda$, its maximal eigenvalue. In this case, $\gamma$ is solution of
\begin{equation} \label{eq:gamma}
\sum_{x \in E} \frac{\gamma(x)}{T(y,x;0)} = \lambda \frac{\gamma(y)}{T(y,y;0) \nu(y)}.
\end{equation}
Define further for any $\eta=(\eta(x) : x \in E) \in \mes{E}$ with full support, the transition matrices $D^{\eta}$ and $U^{\eta}$ from $E$ to $E$:
\begin{equation}\label{eq:DUeta}
D^{\eta}(x;y)=\frac{\displaystyle \sum_{x' \in E} \eta(x') \frac{T(x,x';y)}{T(x,x';0)}}{\displaystyle  \sum_{x'' \in E}\frac{\eta(x'')}{T(x,x'';0)}} \text{ and } 
U^{\eta}(y;x')=\frac{\displaystyle \eta(x') \frac{T(0,x';y)}{T(0,x';0)}}{\displaystyle  \sum_{x'' \in E} \eta(x'') \frac{T(0,x'';y)}{T(0,x'';0)}}
\end{equation}

\begin{theorem}[Theorem~2.6 of~\cite{CM15}] \label{thm:CM15}
Let $\PCA{A}$ be a SPCA with finite alphabet $E = \{0,\dots,\kappa\}$ and transition matrix $T$ such that, for any $x_0,x_1,y_0 \in E$, $T(x_0,x_1;y_0) > 0$. One of the invariant probability measures of $\PCA{A}$ is a HZMC distribution iff $T$ satisfies the two following conditions:
\begin{cond} \label{cond:CM15-1}
for any $x,x',y \in E$, 
\begin{displaymath}
T(x,x';y) T(x,0;0) T(0,x';0) T(0,0;y) = T(0,0;0) T(x,x';0) T(0,x';y) T(x,0;y)
\end{displaymath}
\end{cond} 
\begin{cond} \label{cond:CM15-2}
the equality $D^ \gamma U^\gamma=U^\gamma D^\gamma$ holds (for $\gamma$ as defined in~\eqref{eq:gamma} and $(D^\gamma,U^\gamma)$ in~\eqref{eq:DUeta}).
\end{cond}

In this case, $(D^\gamma,U^\gamma)$-HZMC distribution is invariant by $\PCA{A}$.
\end{theorem} 

We can present, now, two new theorems on PCA that characterize TPCA whose one of its invariant probability measures is a $(D,U)$-HZMC. We establish these characterizations in two particular cases. First case is when $D=U$:
\begin{theorem} \label{thm:3col}
Let $\PCA{A}$ be a TPCA on $E$ a finite alphabet of transition matrix $T = (T(y,x,y';z) : y,x,y',z  \in E)$ with positive rate (i.e. $T(y,x,y';z) >  0$ for any $y,x,y',z \in E$). For any $y,y' \in E$, we denote $\left( \tilde{T}(y,y';x) : x \in E \right)$ the unique left stochastic eigenvector (associated to eigenvalue $1$) of the stochastic matrix $\left(T(y,x,y';z) : x,z \in E\right)$. One of the invariant probability measures of $\PCA{A}$ is a $(D,D)$-HZMC distribution iff the SPCA $\PCA{\tilde{A}}$ on $E$ with transition matrix $\tilde{T} = (\tilde{T}(y,y';z) : y,y',z \in E)$ satisfies \C~\ref{cond:CM15-1} and \C~\ref{cond:CM15-2} of Theorem~\ref{thm:CM15} with $D^\gamma = U^\gamma$. In this case, $(D^\gamma,D^\gamma)$-HZMC distribution is invariant by $\PCA{A}$.
\end{theorem}
The second case is when $E$ is of size $2$:
\begin{theorem} \label{thm:2col}
Let $\PCA{A}$ be a TPCA on $E = \{0,1\}$ of transition matrix $T = (T(y,x,y';z) : y,x,y',z  \in E)$ with positive rate. For any $y,y' \in E$, we denote $\left(\tilde{T}(y,y';x) : x \in E \right)$ the left eigenvector (associated to eigenvalue~$1$) of $\left( \sum_{u}  T(y',x,y;u) T(y,u,y';z) : x,z \in E\right)$. One of the invariant probability measures of $\PCA{A}$ is a $(D,U)$-HZMC distribution iff $\tilde{T}$ satisfies \C~\ref{cond:CM15-1} and 
\begin{cond} \label{cond:2col}
for the pair $(D^\gamma,U^\gamma)$ founded by application of Theorem~\ref{thm:CM15} to SPCA $\PCA{\tilde{A}}$ on E with transition matrix  $\tilde{T}$, we have, for any $y,y',z \in \{0,1\}$,
\begin{displaymath}
D^\gamma(y;z) U^\gamma(z;y') = \sum_{x \in \{0,1\}} U^\gamma(y;x) D^\gamma(x;y') T(y,x,y';z).
\end{displaymath}
\end{cond}
In this case, $(D^\gamma,U^\gamma)$-HZMC distribution is an invariant probability measure of $\PCA{A}$.
\end{theorem}

 Proofs of these two theorems are done in Section~\ref{sec:t3col} and~\ref{sec:t2col}. These two theorems applied to two particular TPCA, $\PCA{A}_8$ and $\PCA{A}_6$ defined in Sections~\ref{sec:PCA8V} and~\ref{sec:PCA6V}, give another way to prove Propositions~\ref{cor:inv2} and~\ref{cor:inv3} (see Section~\ref{sec:p2col}).

\subsection{TPCA $\PCA{A_8}$ and 8-vertex models} \label{sec:PCA8V}
Now, we consider a family of TPCA related to the 8-vertex model when $a+c=b+d$. $\PCA{A_8}$ is a TPCA with alphabet $E=\{0,1\}$ and transition matrix $T$ such that, for any $k \in \{0,1\}$,
\begin{itemize}[topsep=0pt,itemsep=0pt,partopsep=0pt,parsep=0pt]
\item $T(k,k,k;k) = T(k,1-k,k;1-k) = r$,
\item $T(k,k,k;1-k) = T(k,1-k,k,k) = 1-r$,
\item $T(k,1-k,1-k;k) = T(k,k,1-k;1-k) = p$,
\item $T(k,1-k,1-k;1-k) = T(k,k,1-k;k) = 1-p$.
\end{itemize}

To show their links with vertex models, we define first $\overline{K}_\infty$'s faces and coloring of $\overline{K}_\infty$. We call internal faces of $\overline{K}_\infty$ any square whose vertices are $\{(i,t),(i-1/2,t-1/2),(i,t-1),(i+1/2,t-1/2)\}$ for any $(i,t) \in \overline{V}_\infty$ such that $t \neq 0$; such a face is numbered $(i-t,2t)$. And we call external faces any triangle whose vertices are $\{(i-1/2,-1/2),(i,0),(i+1/2,-1/2)\}$ for any $i \in \ZZ$; such a face is numbered $(i,0)$. Set of (internal and external) faces of $\overline{K}_\infty$ is denoted $\overline{F}_\infty$. A 2-coloring of $\overline{K}_\infty$ is any function $C$ from $\overline{F}_\infty$ to $\{0,1\}$. The set of 2-colorings is denoted $\mathcal{C}_2$.

Now we can remark that any realization of the space-time diagram of $\PCA{A_8}$ is a 2-coloring of $\overline{K}_\infty$ (see Figure~\ref{fig:2col}). Baxter\cite[Section~8.13]{Baxter82} presents a function, denoted here $\Theta_8$, from $\mathcal{C}_2$ to $\overline{\Omega}^8_\infty$. This function is the following: starting with any $C \in \mathcal{C}_2$, we obtain an orientation $O=\Theta_8(C) \in \overline{\Omega}^8_\infty$ by the following rule: take any edge $(i,t)$ (this edge is adjacent to 2 faces $f$ and $f'$), the orientation $e(i,t)$ is
\begin{equation} \label{eq:colF}
e(i,t) = \ind{C(f)=C(f')}.
\end{equation}
Conversely, starting with an orientation $O \in \overline{K}_\infty$, we can obtain two distinct 2-colorings $C$ and $C'$ $\in \mathcal{C}_2$ ($\{C,C'\} = \Theta_8^{-1}(\{O\})$) by this way: first, color any face $f$ by any color $0$ or $1$, then color adjacent faces to the previous one respecting~\eqref{eq:colF}, and make it iteratively to color any face. Two distinct 2-colorings $C$ and $C'$, obtained from the same orientation $O$, satisfy the following property: for any face $f$, $C(f) \neq C'(f)$. See Figure~\ref{fig:2col} as an example of $\Theta_8$.\par
First, note that, for any $t$, knowing $C_{|(F_t,F_{t+1})}$, the coloring of the set of faces $\{(i,t') : i \in \ZZ, t'\in \{t,t+1\}\}$, is enough to know the orientations of $(e(i,t):i \in \ZZ) = \Theta_8(C_{|(F_t,F_{t+1})})$. Hence, we can define $\mu_t = \Theta_8(\nu_t)$, a law on $(e(i,t):i \in \ZZ)$ according to $\nu_t$ a law on 2 coloring faces of $\{(i,t') : i \in \ZZ, t'\in \{t,t+1\}\}$ by, for any $n \in \NN$, for any $e_n=(e_{i,t}: i \in \llbracket -n, n \rrbracket = [-n,n] \cap \ZZ )$,
\begin{equation} \label{eq:ext8}
\mu_t(e_n) = \prob{(e(i,t)=e_{i,t} : i \in \llbracket -n, n \rrbracket)} = \sum_{C \in \{C_1,C_2\} = \Theta_8^{-1}(e_n)} \nu_t(C)
\end{equation}

\begin{figure}
\begin{center}
\begin{tabular}{ccc}
\includegraphics{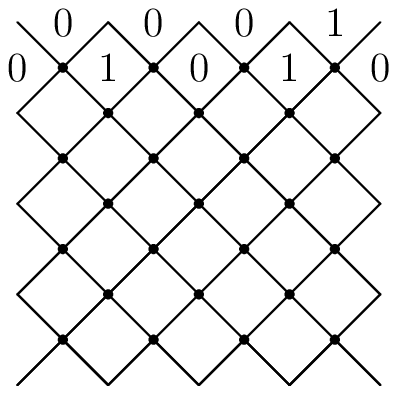} & \includegraphics{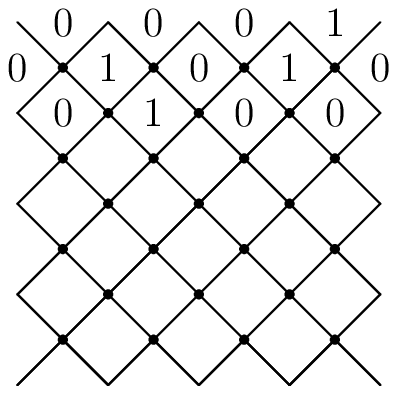} & \includegraphics{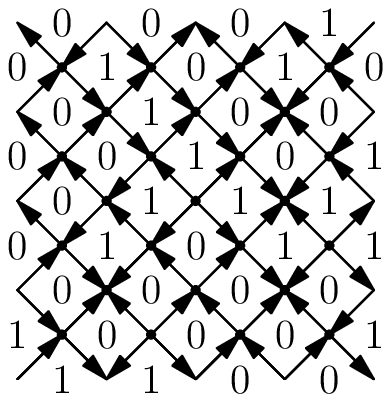} \\
initial state & after an iteration of $A_8$ & final result
\end{tabular}
\end{center}
\caption{One realization of the space time diagram of the TPCA $\PCA{A_8}$ and its associated 8-vertex model configuration.}
\label{fig:2col}
\end{figure}

Now, we can show the reason of our choice for $\PCA{A_8}$.
\begin{lemma} \label{lem:2colPas}
Let $\nu_0$ be any probability measure on $\{0,1\}^\ZZ$ and let $C$ be the space-time diagram of $\PCA{A_8}$ such that $C_{|(F_0,F_1)} \sim \nu_0$, then $\Theta_8(C) \sim \mathcal{L}(\Theta_8(\nu_0);p,r)$. 
\end{lemma}

\begin{proof}
Proofs of these two lemmas are based on the fact that images by $\Theta_8$ of initial laws and transitions of $\PCA{A_8}$ are those that define $\mathcal{L}(\mu;p,r)$ in Definition~\ref{def:loiqpr}.
\end{proof}

Hence, $\PCA{A_8}$ is related to laws $\mathcal{L}(\mu;p,r)$ and so on to the $8$-vertex model. Now, the study of invariant Markovian laws of $\PCA{A_8}$ give us a unique element.
\begin{proposition} \label{prop:2col}
For any $r \in (0,1)$ and $p \in (0,1)$. The set of invariant HZMC of $\PCA{A_8}$ has a unique element that is the $(D,U)$-HZMC whose kernels $D$ and $U$ are such that $D=U$ and, for any $i,j \in \{0,1\}$, $D(i;j) = 1/2$.
\end{proposition}

Proof of this proposition is done in Section~\ref{sec:p2col}. Proposition~\ref{cor:inv2} is, then, an immediate consequence of Lemma~\ref{lem:2colPas} and of this proposition. In addition, 
\begin{lemma} \label{lem:ergoFace}
For any $r \in (0,1)$ and $p \in (0,1)$, $\PCA{A_8}$ is ergodic: for any initial law $\nu_0$, let $C$ be the space-time diagram of $\PCA{A_8}$ such that $C_{|(F_0,F_1)} \sim \nu_0$ and denote $\nu_t$ the law of $C_{|(F_t,F_{t+1})}$, then $\nu_t \to \PM{1/2}$ as $t \to \infty$.
\end{lemma}
Proof of this lemma is done in a more general context in an incoming paper on triangular probabilistic cellular automata~\cite{CM17}. This lemma permits to prove Proposition~\ref{prop:ergo}.

\begin{proof}[Proof of Proposition~\ref{prop:ergo}]
Let $\mu$ be any law on $\{0,1\}^\ZZ$. Now, we have the choice for our initial law $\nu_0$ on coloring. We choose here the one that is symmetric: for any $n$, for any $C = (c_i : i \in \llbracket -2n-1,2n \rrbracket) \in \{0,1\}^{{\llbracket -n,n \rrbracket} \times {\llbracket -n-1,n \rrbracket}}$,
\begin{equation}
\nu_0(C) = \nu_0((1- c_i : i \in \llbracket -2n-1,2n \rrbracket)) = \frac{1}{2} \mu(\Theta_8(C)). 
\end{equation}
Now, by Lemma~\ref{lem:ergoFace}, $\nu_t \to \PM{1/2}$. And, so, $\mu_t = \Theta_8(\nu_t) \to \PM{1/2}$, that is the law of $(e(i,t):i \in \ZZ)$.
\end{proof}

\subsection{TPCA $\PCA{A_6}$ and 6-vertex models} \label{sec:PCA6V}
In this section, we show relations between 6-vertex model when $a+c=b$ and TPCA $\PCA{A_6}$. $\PCA{A_6}$ is a TPCA with alphabet $E=\{0,1,2\}$ and transition matrix $T$ such that, for any $i \in \{0,1,2\}$, 
\begin{itemize}[topsep=0pt,itemsep=0pt,partopsep=0pt,parsep=0pt]
\item $T(i,i+1,i+2 ;i+1) = 1$,
\item $T(i,i+1,i;i+2) = p$,
\item $T(i,i+1,i;i+1) = 1-p$
\end{itemize}
where additions on $E$ are done modulo 3. 

Links between $\PCA{A_6}$ and 6-vertex model are similar to the ones between $\PCA{A_8}$ and 8-vertex model, instead of that 2-coloring is replaced by proper 3-coloring.
A proper 3-coloring of $\overline{K}_\infty$ is any function $C$ from $\overline{F}_\infty$ to $\{0,1,2\}$ such that if two different faces $f,f'$ have a common edge then $C(f) \neq C(f')$. The set of proper 3-colorings is denoted $\mathcal{C}_3$.

We can remark that if we start iterations of $\PCA{A_6}$ with an initial state $(S_0,S_1)$ such that, for any $i \in \ZZ$, $S_{0}(i) \neq S_{1}(i)$ and $S_{1}(i) \neq S_{0}(i+1)$ a.s., then the same condition is satisfied for any $t \geq 0$, i.e., for any $t \in \NN$, for any $i \in \ZZ$, $S_{t}(i) \neq S_{t+1}(i)$ and $S_{t+1}(i) \neq S_{t}(i+1)$ a.s. Hence, a space-time diagram realization of $\PCA{A_6}$ is a proper $3$-coloring of $\overline{K}_\infty$. Moreover, there exists a function, denoted here $\Theta_6$, between $\mathcal{C}_3$ and $\overline{\Omega}^6_\infty$~\cite[Section~8.13]{Baxter82}. This function is: let $C \in \mathcal{C}_3$, take any edge $(i,t)$ of $\overline{K}_\infty$, edge $(i,t)$ is oriented such that if we look the oriented edge in front of us oriented to the top, then the value of the right face of the edge is equal (modulo 3) to the value of the left face +1 (see Figure~\ref{fig:cont3}). Conversely, starting with an orientation $O \in \overline{\Omega}^6_\infty$, we get three distinct proper 3-colorings $\{C,C',C''\}= \Theta_6^{-1}(\{O\})$. These three distinct 3-colorings $C$, $C'$ and $C''$  satisfy: for any face $f \in \overline{F}_\infty$, $\{C(f),C'(f),C''(f)\} = \{0,1,2\}$.\par

\begin{figure}
\begin{center}
\begin{tabular}{cccc}
\includegraphics{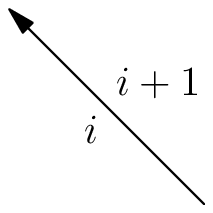} & \includegraphics{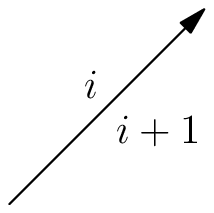} & \includegraphics{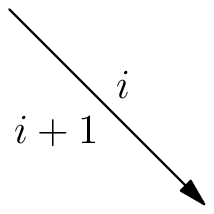} & \includegraphics{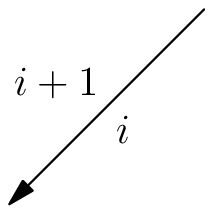}
\end{tabular}
\end{center}
\caption{Relations observed by $\Theta_6$ permitting to go from a proper 3-coloring of $\overline{K}_\infty$ to a configuration of the 6-vertex model.}
\label{fig:cont3}
\end{figure}

In a similar way that has been done in~\eqref{eq:ext8}, we can define $\Theta_6$ as a function on measure of $\{0,1,2\}^\ZZ$.
\begin{lemma} \label{lem:3colPas}
Let $\mu$ be any measure on $\{0,1,2\}^\ZZ$ such that if $(S_0,S_1) \sim \mu$ then for any $i \in \ZZ$, $S_{0}(i) \neq S_{1}(i)$ and $S_{1}(i) \neq S_{0}(i+1)$ a.s. Let $C$ be the space-time diagram of $\PCA{A_6}$ such that $(S_0,S_1) \sim \mu$, then $\Theta_6(C) \sim \mathcal{L}(\Theta_6(\mu);p,1)$.
\end{lemma}

\begin{proof}
Images by $\Theta_6$ of initial laws and transition of $\PCA{A_6}$ are those that define $\mathcal{L}(\Theta_6(\mu);p,1)$.
\end{proof}

An interesting property of $\PCA{A_6}$ is:
\begin{proposition} \label{prop:3col}
For any $p \in [0,1]$.  The set of invariant HZMC of $\PCA{A_6}$ contains the set of $(D,U)$-HZMC whose kernels $D$ and $U$ are such that $D=U$ and, for any $i \in \{0,1,2\}$, $D(i;i+1 \text{ mod } 3) = q$ and $D(i;i-1 \text{ mod } 3) = 1-q$ for any $q \in [0,1]$.
\end{proposition}

This property associated to Lemma~\ref{lem:3colPas} permits to get an alternative proof of Proposition~\ref{cor:inv3}. Proof of Proposition~\ref{prop:3col} is done in Section~\ref{sec:p2col}.

\subsection{Proofs of previous results on  TPCA} \label{sec:PACP}
\subsubsection{Preliminary results on TPCA and invariant HZMC distributions} \label{sec:LACP}
First of all, we recall necessary and sufficient conditions for a $(D,U)$-HZMC to be an invariant probability measure of a SPCA.
\begin{proposition}[Proposition~1.2 of~\cite{CM15}] \label{prop:CM15}
Let $E$ be a finite set. Let $\PCA{A}$ be a PCA with positive rate and transition matrix $T$ and $(D,U)$ be two transition matrices from $E$ to $E$. The $(D,U)$-HZMC distribution is an invariant probability measure of $\PCA{A}$ iff the two following conditions hold:
\begin{cond} \label{cond:CM15T}
for any $x,x',y \in E$,
\begin{equation} \label{eq:CM15T}
T(x,x';y) = \frac{D(x;y) U(y;x')}{(DU)(x;x')}
\end{equation}
\end{cond}
and
\begin{cond} \label{cond:CM15DU}
\begin{equation} \label{eq:CM15DU}
DU = UD.
\end{equation}
\end{cond}
\end{proposition}
This proposition is weaker than Theorem~\ref{thm:CM15} in the sense that \C~\ref{cond:CM15T} and~\C~\ref{cond:CM15DU} hold both on $T$ and $(D,U)$ and not just only on $T$. First step to prove Theorems~\ref{thm:3col} and~\ref{thm:2col} is to generalize this proposition to TPCA.
\begin{lemma} \label{lem:inv}
Let $E$ be a finite set. Let $\PCA{A}$ be a TPCA of transition matrix $T$ with positive rate and let $D$ and $U$ be two transition matrices from $E$ to $E$. The $(D,U)$-HZMC distribution is an invariant probability measure of $\PCA{A}$ iff 
\begin{cond} \label{cond:inv}
for any $y,y',z \in E$,
\begin{equation} \label{eq:inv}
D(y;z) U(z;y') = \sum_{x \in E} U(y;x) D(x;y') T(y,x,y';z).
\end{equation}
\end{cond}
\end{lemma}

\begin{proof}
Let $\PCA{A}$ be a TPCA of t.m.\ $T$ with positive rate and $(D,U)$ two t.m.\ from $E$ to $E$.
 
$\bullet$  Suppose that the $(D,U)$-HZMC distribution is an i.p.m.\ of $\PCA{A}$. Suppose that a pair of lines $(x_0,x_1) \sim \HZ{(D,U)}$, then $(x_1,x_2) \sim \HZ{(D,U)}$ where $x_2$ is the image of $(x_0,x_1)$ by $\PCA{A}$. Now, for any $a,c,d \in E$,  we compute $\prob{x_1(0) = a, x_2(0) = d, x_1(1) = c | x_1(0)=a}$. On one hand,  the lines $(x_1,x_2) \sim \HZ{(D,U)}$, so
\begin{equation} \label{eq:truc1}
\prob{x_1(0) = a, x_2(0) = d, x_1(1) = c | x_1(0)=a} = D(a;d) U(d;c)
\end{equation}
and, on the other hand, the pair $(x_0,x_1) \sim \HZ{(D,U)}$ and $x_2$ is their image by $\PCA{A}$, so
\begin{align} 
& \prob{x_1(0) = a, x_2(0) = d, x_1(1) = c | x_1(0)=a} \nonumber \\
& = \sum_{b \in E} \prob{x_1(0) = a, x_0(1) = b, x_2(0) = d, x_1(1) = c | x_1(0)=a} \nonumber \\
& = \sum_{b \in E} \prob{x_1(0) = a, x_0(1) = b, x_2(0) = d | x_1(0)=a} \prob{x_1(1) = c | x_1(0)=a, x_0(1)=b, x_1(1)=c} \nonumber \\
& = \sum_{b \in E} U(a;b) D(b;c) T(a,b,c;d). \label{eq:truc2}
\end{align}
By~\eqref{eq:truc1} and~\eqref{eq:truc2}, we finally obtain that \C~\ref{cond:inv} is necessary.

$\bullet$ Reversely, suppose that \C~\ref{cond:inv} holds, and take a pair of lines $(x_0,x_1) \sim \HZ{(D,U)}$ and $x_2$ their image by $\PCA{A}$. Then, for any $k_1,k_2 \in \ZZ$, $k_1 < k_2$ for any $(b_i : k_1 \leq i \leq k_2) \in E^{k_2-k_1+1}$ and $(c_i : k_1 \leq i \leq k_2-1) \in E^{k_2-k_1}$, 
\begin{align}
& \prob{ (x_1(i) = b_i : k_1 \leq i \leq k_2) , (x_2(i) = c_i : k_1 \leq i \leq k_2-1) } \nonumber \\
& = \sum_{a_i \in E : k_1 \leq i \leq k_2+1} \rho_{k_1}(a_{k_1}) \left( \prod_{i=k_1}^{k_2} D(a_i;b_i) U(b_i;a_{i+1}) \right) \left( \prod_{i=k_1}^{{k_2}-1} T(b_i,a_{i+1},b_{i+1};c_i) \right) \\
& = \left(\sum_{a_{k_1} \in E} \rho_{k_1}(a_{k_1}) D(a_{k_1};b_{k_1}) \right) \prod_{i=k_1}^{{k_2}-1} \sum_{a_{i+1} \in E} U(b_i;a_{i+1})  D(a_{i+1};b_{i+1}) T(b_i,a_{i+1},b_{i+1};c_i)  \\
& = \rho_{k_1}(b_{k_1}) \prod_{i=k_1}^{k_2-1} D(b_i;c_i)  U(c_i;b_{i+1})
\end{align}
Then, the pair $(x_1,x_2) \sim \HZ{(D,U)}$.
\end{proof}

\begin{remark} \label{rem:inv}
When the transition matrix has not positive rates, \C~\ref{cond:inv} implies always that the $(D,U)$-HZMC is an invariant probability measure of $\PCA{A}$, but reverse is not true because \C~\ref{cond:inv} can hold on a subset of $E$, but not $E$ entirely.
\end{remark}

We continue proving Theorems~\ref{thm:3col} and~\ref{thm:2col} by seeing that, in Prop~\ref{prop:CM15}, for any $(D,U)$-HZMC, there exists a unique SPCA $\PCA{A^S}$ that lets the $(D,U)$-HZMC invariant, its transition matrix $T^{S}$ is, for any $y,y',z \in E$,
\begin{equation} \label{eq:TS}
T^{S}(y,y';z) = \frac{D(y;z)U(z;y')}{(DU)(y;y')}.
\end{equation}

For the same reason, there exists a unique SPCA $\PCA{A}^R$ that lets the $(U,D)$-HZMC invariant, its transition matrix $T^R$ is, for any $y,y',x \in E$,
\begin{equation} \label{eq:TR}
T^{R}(y,y';x) = \frac{U(y;x)D(x;y')}{(DU)(y;y')}.
\end{equation}

Then, \C~\ref{cond:inv} is equivalent, dividing by $(DU)(y;y')$ (not equal to zero in positive rates cases), to 
\begin{cond} \label{cond:inv2}
for any $y,y',z$,
\begin{equation} \label{eq:inv2}
T^{S}(y,y';z) = \sum_{x \in E} T^{R}(y,y';x) T(y,x,y';z).
\end{equation}
\end{cond}

\begin{corollary} \label{cor:TSPCA}
Let $E$ be a finite set. Let $\PCA{A}$ be a TPCA of transition matrix $T$ with positive rates and let $D$ and $U$ be two transition matrices from $E$ to $E$. The $(D,U)$-HZMC is an invariant probability measure of $\PCA{A}$ iff \C~\ref{cond:inv2} is satisfied with $T^S$ the transition matrix of the unique SPCA $\PCA{A^S}$ that lets the $(D,U)$-HZMC invariant and $T^R$ the transition matrix of the unique SPCA $\PCA{A^R}$ that lets the $(U,D)$-HZMC invariant.
\end{corollary}

The main idea to prove Theorems~\ref{thm:3col} and~\ref{thm:2col} is to find, for a fixed transition matrix $T$ from $E^3$ to $E$, all the pair of transition matrices $(T^S,T^R)$ from $E^2$ to $E$ such that \C~\ref{cond:inv2} is satisfied, and then verify if $(T^S,T^R)$ satisfies (or not) the other wanted properties: conservation of a $(D,U)$-HZMC and of a $(U,D)$-HZMC thanks to Theorem~\ref{thm:CM15}. In the particular cases where $D=U$ or $E=\{0,1\}$, we are able to find a unique possible pair of $(T^S,T^R)$ related to $T$ that can satisfy~\C~\ref{cond:inv2}. All other cases are open problems.

\subsubsection{Proof of Theorem~\ref{thm:3col}} \label{sec:t3col}
Let $T$ be a t.m.\ from $E^3$ to $E$ of a TPCA with positive rate. We denote, for any $y,y'$, $(\tilde{T}(y,y';x) : x \in E)$, the unique left eigenvector related to the eigenvalue~$1$ of $\left(T(y,x,y';z) : x \in E, z \in E\right)$ normalized such that $\sum_{x\in E} \tilde{T}(y,y';x) = 1$, i.e.\ for any $y,y'$,
\begin{equation} \label{eq:p3col}
\tilde{T}(y,y';z) = \sum_{x \in E} \tilde{T}(y,y';x) T(y,x,y';z)
\end{equation}
and $\tilde{T}$ is a t.m.\ from $E^2$ to $E$, this eigenvector exists due to Perron-Frobenius theorem.
Moreover, we suppose that $\tilde{T}$ satisfies \C~\ref{cond:CM15-1} and~\C~\ref{cond:CM15-2} of Theorem~\ref{thm:CM15} with $D^\eta=U^\eta$, i.e.\ there exists $D^\eta$ such that the $(D^\eta,D^\eta)$-HZMC distribution is an i.p.m.\ of $\PCA{\tilde{A}}$, the SPCA with t.m.\ $\tilde{T}$. In this case, we remark that SPCAs that let invariant $(D,U)$-HZMC and $(U,D)$-HZMC are the same, i.e.\ $T^R = T^S = \tilde{T}$ in Corollary~\ref{cor:TSPCA} and so \eqref{eq:p3col} imply~\C~\ref{cond:inv2}. We finish the proof using Corollary~\ref{cor:TSPCA}.\par
\smallskip
Reversely, if the $(D,D)$-HZMC distribution is an i.p.m.\ of $T$, by Lemma~\ref{lem:inv}, for any $y,y',z \in E$,
\begin{displaymath}
D(y;z) D(z;y') = \sum_{x \in E} D(y;x) D(x;y') T(y,x,y';z),
\end{displaymath}
i.e.\ for any $y,y'$, $\left(D(y;x)D(x;y') : x\in E\right)$ is a left eigenvector of $\left(T(y,x,y';z) : x \in E, z \in E\right)$ associated to the eigenvalue~$1$. By Perron-Frobenius theorem, the eigenspace associated to eigenvalue~$1$ is of dimension~$1$. And, so, for any $y,y',x$, $D(y;x)D(x;y') = \lambda_{y,y'} \tilde{T}(y,y';x)$. Moreover, as we want that $(\tilde{T}(y,y';x) : x \in E)$ is a probability vector, we obtain
\begin{displaymath}
\tilde{T}(y,y';x) = \frac{D(y;x)D(x;y')}{(DD)(y;y')}.
\end{displaymath}
Hence, by Proposition~\ref{lem:inv}, the $(D,D)$-HZMC distribution is an i.p.m.\ of SPCA $\PCA{\tilde{A}}$ of t.m.\ $\tilde{T}$. And so, by Theorem~\ref{thm:CM15}, $\PCA{\tilde{A}}$ needs to satisfy \C~\ref{cond:CM15-1} and~\C~\ref{cond:CM15-2} with $D=U$.

\subsubsection{Proof of Theorem~\ref{thm:2col}} \label{sec:t2col}
In the case $E=\{0,1\}$, we have the following algebraic property on $T^S$ and $T^R$.
\begin{lemma} \label{lem:2col}
Let $E=\{0,1\}$ and $T^S$ and $T^R$ two transition matrices from $E^2$ to $E$ such that the $(D,U)$-HZMC is an invariant probability measure of $\PCA{A^S}$ with transition matrix $T^S$ and the $(U,D)$-HZMC is an invariant probability measure of $\PCA{A^R}$ with transition matrix $T^R$. Then, for any $y,y',x \in E$,
\begin{equation} \label{eq:TRS}
T^S(y,y';x) = T^R(y',y;x).
\end{equation}
\end{lemma}

\begin{proof}
As $\PCA{A^S}$ lets invariant the $(D,U)$-HZMC distribution, by Pro\-position~\ref{prop:CM15}, for any $y,x,y'$, \eqref{eq:TS} holds and, due to similar reason, for any $y,x,y'$, \eqref{eq:TR} holds too.\par
When $y=y'=x$, by~\eqref{eq:TS} and~\eqref{eq:TR}, $T^S(y,y;y) = T^R(y,y;y)$, and, moreover, as $T^S(y,y;.)$ and $T^R(y,y;.)$ are probability measures on $\{0,1\}$, \eqref{eq:TRS} holds if $y=y'$.\par
Now, we look the more complicated case $y=0$ and $y'=1$ (case $y=1$ and $y'=0$ is similar replacing $S$ by $R$) and $x=0$ ($x=1$ will be then immediate because $T^S(0,1;0) + T^S(0,1;1) = 1 = T^R(1,0;0) + T^R(1,0;1)$). By Proposition~\ref{prop:CM15}, $DU=UD$, so $(UD)(0;0) = (DU)(0;0)$ that simplifies in $U(0;1) D(1;0) = D(0;1) U(1;0)$, that implies $U(0;1) (UD)(1;0) = U(1;0) (DU)(0;1)$ and, finally,
\begin{displaymath}
T^S(0,1;0) = \frac{D(0;0) U(0;1)}{(DU)(0;1)} = \frac{U(1;0) D(0;0)}{(UD)(1;0)} = T^R(1,0;0).
\end{displaymath}
\end{proof}

Now, we can prove Theorem~\ref{thm:2col}.\par
\begin{proof}[Proof of Theorem~\ref{thm:2col}] 
$\bullet$ Suppose that one of the i.p.m.\ of TPCA $\PCA{A}$ of t.m.\ $T$ with positive rates is a $(D,U)$-HZMC distribution and denote $\left( \tilde{T}(y,y';x) : x \in E\right)$ the left eigenvector related to the eigenvalue~$1$ of $\left(\sum_{k \in E} T(y',x,y;k) T(y,k,y';z) : x \in E, z \in E\right)$ and such that $\sum_{x \in E} T(y,y';x)=1$.\par By Lemma~\ref{lem:inv}, \C~\ref{cond:inv} holds. As $E=\{0,1\}$, by Lemma~\ref{lem:2col}, this condition rewrites: for any $y,y',z$,
\begin{equation}
T^{S}(y,y';z) = \sum_{x \in E} T^{S}(y',y;x) T(y,x,y';z).
\end{equation}
Applying this equation twice establishes that, for any $y,y',z \in \{0,1\}$,
\begin{equation}
\sum_{x \in E} T^{S}(y,y';x) \left(\sum_{u \in E} T(y',x,y;u) T(y,u,y';z) \right) = \sum_{x \in E} T^S(y',y;x) T(y,x,y';z) = T^S(y,y';z).
\end{equation}
In other words, for any $y,y'$, $\left(T^S(y,y';x) : x \in E \right)$ is a left eigenvector related to eigenvalue $1$ of $\left(\sum_{u \in E} T(y',x,y;u) T(y,u,y';z) : x \in E, z \in E \right)$. So, by Perron-Frobenius theorem, for any $y,x,y'$, $T^S(y,y';x) = \lambda_{y,y'} \tilde{T}(y,y';x)$ with $\lambda_{y,y'} =1$ because the sum in $x$ is equal to $1$ in both sides. Then, the TPCA $\PCA{\tilde{A}}$ with t.m.\ $\tilde{T} = T^S$ lets invariant the $(D,U)$-HZMC distribution. Hence, by Theorem~\ref{thm:CM15} or results of Belyaev~\cite{BGM69}, $\tilde{T}$ satisfies
\begin{equation} \label{eq:Belyaev}
\tilde{T}(0,0;0) \tilde{T}(0,0;1) \tilde{T}(1,0;0) \tilde{T}(0,1;0)  = \tilde{T}(1,1;1) \tilde{T}(1,1;0) \tilde{T}(0,1;1) \tilde{T}(1,0;1) 
\end{equation}
And so \C~\ref{cond:2col} holds by Lemma~\ref{lem:inv}. 

$\bullet$ Reversely, if $\tilde{T}$ satisfies \C~\ref{cond:CM15-1} (i.e.~\eqref{eq:Belyaev} if $E=\{0,1\}$), then we apply Theorem~\ref{thm:CM15} to find a pair $(D^\eta,U^\eta)$ such that the $(D^\eta,U^\eta)$-HZMC distribution is an i.p.m.\ of $\PCA{\tilde{A}}$ of t.m.\ $\tilde{T}$. If, moreover, this pair $(D^\eta,U^\eta)$ satisfies \C~\ref{cond:2col} then, by Lemma~\ref{lem:inv}, the $(D^\eta,U^\eta)$-HZMC is an i.p.m.\ of $\PCA{A}$.
\end{proof}

\subsubsection{Proofs of Proposition~\ref{prop:2col} and~\ref{prop:3col}} \label{sec:p2col}
\begin{proof}[Proof of Proposition~\ref{prop:2col}]
To prove Proposition~\ref{prop:2col}, we apply Theorem~\ref{thm:2col} to $\PCA{A_8}$. First, let us compute matrices 
\begin{displaymath}
\left(\sum_{k \in E} T(y',x,y;k) T(y,k,y';z) : x \in E, z \in E\right)
\end{displaymath}
for any $y,y'$. We obtain the following four matrices:
\begin{displaymath}
\begin{array}{c|c|c}
y \backslash y' & 0 & 1 \\
\hline
0 & \begin{pmatrix}
p^2 + (1-p)^2 & 2p(1-p) \\
2p(1-p) & p^2 + (1-p)^2
\end{pmatrix} &
\begin{pmatrix}
r^2 + (1-r)^2 & 2r(1-r) \\
2r(1-r) & r^2 + (1-r)^2
\end{pmatrix} \\
\hline
1 & 
\begin{pmatrix}
r^2 + (1-r)^2 & 2r(1-r) \\
2r(1-r) & r^2 + (1-r)^2
\end{pmatrix} &
\begin{pmatrix}
p^2 + (1-p)^2 & 2p(1-p) \\
2p(1-p) & p^2 + (1-p)^2
\end{pmatrix}
\end{array}
\end{displaymath}

Left eigenvectors related to eigenvalue $1$ are all equal for these four matrices and their common value is $\begin{pmatrix} 1/2 & 1/2 \end{pmatrix}$. So, we know have to study the SPCA whose t.m.\ is, for any $y,x,y' \in \{0,1\}$, $T(y,y';x) = 1/2$. We observe easily that this SPCA as for unique invariant HZMC distribution, the $(D,U)$-HZMC such that, for any $x,y \in \{0,1\}$, $D(x;y) = U(x;y)= 1/2$. Then, \C~\ref{cond:2col} holds for this pair $(D,U)$. We deduce, by Theorem~\ref{thm:2col}, that $\PCA{A_8}$ lets invariant this $(D,U)$-HZMC distribution and, moreover,  it is the unique HZMC distribution that is invariant by $\PCA{A_8}$.
\end{proof}

\begin{proof}[Proof of Proposition~\ref{prop:3col}]
To prove Proposition~\ref{prop:3col}, we check that \C~\ref{cond:inv} holds with $T$, the t.m.\ of $\PCA{A_6}$, and for any $(D,U)$ such that $D=U$ and, for any $i \in \ZZ/3\ZZ$, $D(i,i+1) = 1 - D(i,i-1) = q$. And, then, Remark~\ref{rem:inv} concludes the proof.
\end{proof}

\section{Conclusion} \label{sec:concl}
We have computed the edge correlation function of the 8-vertex model on $\overline{K}_\infty$ with free boundary conditions and $a+c=b+d$ and we have bounded the influence of being not in a free boundary conditions case.\par
Moreover, as stated in Proposition~\ref{prop:SKN}, edge correlation function of Theorem~\ref{thm:Cit} is the one of the 8-vertex model on $K_N$ with $a+c=b+d$ and free boundary conditions. If, instead of a rotation of an angle $-\pi/4$ to pass from the 8-vertex model on $K_N$ to the 8-vertex model on $\overline{K}_N$, we have done a rotation by an angle $\pi/4$, then we would have obtained the correlation function of the 8-vertex model on $K_N$ with $a+d = b+c$ and free boundary conditions.

\medskip
\bibliographystyle{abbrv}

\end{document}